\documentclass[10pt]{amsart}

\usepackage{amssymb,amsmath,amsthm,amscd}

\advance\oddsidemargin by -1cm
\advance\evensidemargin by -1cm
\newtheorem{theorem}{Theorem}

\newtheorem{proposition}{Proposition}

\newtheorem{lemma}{Lemma}

\theoremstyle{definition}
\newtheorem{remark}{Remark}

\newcommand{\bdm}{\begin{displaymath}}
\newcommand{\edm}{\end{displaymath}}
\newcommand{\bq}{\begin{equation}}
\newcommand{\eq}{\end{equation}}
\newcommand{\bqn}{\begin{equation*}}
\newcommand{\eqn}{\end{equation*}}

\newcommand{\rn}{\mathbb{R}^n}

\newcommand{\eps}{\varepsilon}

\newcommand{\phw}{\tilde \psi^{wk}}
\newcommand{\Sing}{\mathrm{Sing}\,}
\newcommand{\Reg}{\mathrm{Reg}\,}

\newcommand{\norm}[1]{\left\| #1 \right\|}
\newcommand{\mklm}[1]{\left\{ #1 \right\}}
\newcommand{\eklm}[1]{\left\langle #1 \right\rangle}

\renewcommand{\d}{\,d}
\newcommand{\N}{{\mathbb N}}

\newcommand{\R}{{\mathbb R}}

\newcommand{\T}{{\rm T}}

\renewcommand{\epsilon}{\varepsilon}
\renewcommand{\phi}{\varphi}
\renewcommand{\rho}{\varrho}

\newcommand{\Cinft}{{\rm C^{\infty}}}
\newcommand{\CT}{{\rm C^{\infty}_c}}

\newcommand{\g}{{\bf \mathfrak g}}

\newcommand{\id}{\mathrm{id}\,}

\renewcommand{\det}{\mathrm{det}\,}

\newcommand{\vol}{\text{vol}\,}

\newcommand{\Crit}{\mathrm{Crit}}

\DeclareMathOperator{\supp}{supp}

\DeclareMathOperator{\tr}{tr}
\DeclareMathOperator{\gd}{\partial}

\begin{document}

\author{Pablo Ramacher}
\title[Singular equivariant asymptotics and the moment map I]{Singular equivariant asymptotics and the moment map I  } 
\address{Pablo Ramacher, Georg-August-Universit\"at G\"ottingen, Institut f\"ur Mathematik, Bunsenstr. 3-5, 37073 G\"ottingen, Germany}
\email{ramacher@uni-math.gwdg.de}
\thanks{This research was financed by the grant RA 1370/2-1 of the German Research Foundation (DFG)}

\maketitle

\setcounter{tocdepth}{1}
\tableofcontents

\section{Introduction}

This is the first of a series of papers dealing with the asymptotic behavior of certain integrals occuring in the description of the spectrum of an invariant elliptic operator on a compact Riemannian manifold $M$ carrying the action of a compact, connected Lie group of  isometries $G$, and in the study of its equivariant cohomology via the moment map  $\mathbb{J}:T^\ast M \rightarrow \g^\ast$, where $T^\ast M$ and $\g$ denote the cotangent bundle of $M$ and the Lie algebra of $G$, respectively. In the latter context \cite{duistermaat-heckman, atiyah-bott84, witten92, berline-getzler-vergne},  the mentioned integrals are of the form
\bqn 
I(\mu)=\int_{T^\ast M\times \g} e^{i\mathbb{J}(m)(X)/\mu} a(m,X) \d m \d X,  \qquad \mu \to 0^+,
\eqn
where $dm$ is a density on $T^\ast M$, $dX$ is the Lebesgue measure in $\g$, and $a \in \CT(T^\ast M \times \g)$ is an amplitude. While asymptotics for $I(\mu)$ have been obtained for free group actions, one meets with serious difficulties when singular orbits are present. The reason is that, when trying to examine these integrals via the generalized stationary phase theorem in the case of general effective actions, the critical set of the phase function $\mathbb{J}(m)(X)$ is no longer a smooth manifold, so that, a priori, the principle of the stationary phase can not be applied in this case. Nevertheless, in what follows, we shall show how to circumvent this obstacle by partially resolving the singularities of the critical set of $\mathbb{J}(m)(X)$, and in this way obtain asymptotics for $I(\mu)$   with  remainder estimates  in the case of singular group actions. We shall restrict ourselves first to  orthogonal group actions in $\rn$, while the global theory, together with the applications, shall be treated in a second paper.

\section{Compact group actions and the moment map}

From now on let $G$ be  a compact, connected Lie group  with Lie algebra $\g$ acting orthogonally on Euclidean space $\rn$. Note that any finite-dimensional, separable metric $G'$-space with only finitely many orbit types, where $G'$ is a compact Lie group, may be embedded equivariantly in an orthogonal action of $G'$ on some Euclidean space, see Bredon  \cite{bredon}, Section II.10. The considered type of group actions is therefore already quite general. Consider now the cotangent space $T^\ast \rn\simeq \R^{2n}$, endowed  with the canonical coordinates $(x_1,\dots, x_n, \xi_1,\dots \xi_n)$. It constitutes a symplectic manifold whose symplectic form is given by 
\bqn 
\omega= d\theta =\sum _{i=1}^n \d \xi_i  \wedge \, \d x_i,
\eqn
where $\theta=\sum \xi_i \d x_i$ is the Liouville form. The group $G$ acts on $\T^\ast \rn$ by $g (x,\xi)=(g \, x, g\, \xi)$ in a Hamiltonian way, and if we denote by $\tilde X$ the fundamental vector field generated by an element $X$ of $\g$, the corresponding moment map is given by 
\bqn
\mathbb{J}:T^\ast\rn\simeq {\rn}\times \rn\to \g^\ast,  \quad \mathbb{J}(x,\xi)(X)=\theta(\tilde X)(x,\xi)=\eklm{ X x, \xi} ,
\eqn
where $\eklm{\cdot,\cdot}$ is the Euclidean inner product in $\rn$. We are interested in the asymptotic behavior of integrals of the form 
\bq
\label{int}
I(\mu)=   \int _{T^\ast \rn}  \int_{\g} e^{i \psi( x,\xi,X)/\mu }   a(x,\xi,X)  \d X \d\xi \d x , \qquad \mu \to 0^+,   
\eq 
where $dxd\xi$, and $dX$ are Lebesgue measures in $T^\ast \rn$,  and $\g$, respectively, $a \in \CT( T^\ast \rn \times \g)$, and
\bqn 
\psi(x,\xi,X) = \mathbb{J}(x,\xi)(X)= \eklm{ X x,\xi}.
\eqn 
We would like to  study the integrals $I(\mu)$ by means of the generalized stationary phase theorem, and for this we have to consider the critical set of the phase function $\psi(x,\xi,X)$ given by
 \begin{align*}
 \Crit(\psi)&=\mklm{ (x,\xi,X) \in {T^\ast \rn} \times \g: \psi_\ast  (x,\xi,X) =0}=\mklm{(x,\xi,X)\in \Omega \times \g: X \in \g_{(x,\xi)}},
\end{align*}
    where
\bqn
  \Omega=\mathbb{J}^{-1}(0)= \mklm{(x,\xi) \in {T^\ast \rn}: \eklm {Ax,\xi}=0 \text{ for all } A \in \g}
\eqn
represents the zero level of the moment map, and 
\bqn 
\g_{(x,\xi)} = \mklm{ X\in \g: Xx=0, \quad X\xi=0}
\eqn
denotes the Lie algebra of the isotropy group $G_{(x,\xi)}$ statilizing the point $(x,\xi)$. Now, the major difficulty in applying the generalized stationary phase theorem in our setting  stems from the fact that, due to the singular orbit structure of the underlying group action,  the zero level $\Omega$ of the moment map, and, consequently, the considered critical set $\Crit(\psi)$, are in general singular varieties. In fact,  if the $G$-action on $T^\ast\rn$ is not free, the considered moment map is no longer a submersion, so that $\Omega$ and the symplectic quotient $\Omega/G$ are not smooth anymore. Nevertheless, it can be shown that these spaces have Whitney stratifications into smooth submanifolds, see \cite{lerman-sjamaar}, and 
Ortega-Ratiu \cite{ortega-ratiu}, Theorems 8.3.1 and 8.3.2, which correspond to the stratifications of $T^\ast\rn$, and $\rn$ by orbit types, see Duistermaat-Kolk \cite{duistermaat-kolk}. In particular, $\Omega$ has a principal stratum given by
\bqn
\mathrm{Reg} \, \Omega =\mklm{ (x,\xi) \in \Omega: G_{(x,\xi)} \text{ is of principal type}},
\eqn
which is an open and dense subset of $\Omega$, see Cassanas-Ramacher \cite{cassanas-ramacher}, Proposition 2. In addition, $\mathrm{Reg} \,\Omega$ is a smooth submanifold in $\R^{2n}$ of codimension equal to the dimension $\kappa$ of a principal orbit. It is then clear that the smooth part of $\Crit(\psi)$ is given by 
\bqn
\mathrm{Reg}\,  \Crit(\psi)=\mklm{(x,\xi,X)\in \mathrm{Reg}\, \Omega \times \g: X \in \g_{(x,\xi)}},
\eqn
and constitutes a submanifold of codimension $2\kappa$.
To obtain an asymptotic description of $I(\mu)$, we shall partially resolve the singularities of $\Crit(\psi)$, for which we will need a suitable $G$-invariant covering of $\rn$. More generally, and following Kawakubo \cite{kawakubo}, Theorem 4.20,  we shall construct such a covering for an arbitrary compact Riemannian $K$-manifold $M$, where $K$ is a compact, connected Lie group of isometries. Thus, let $(H_1), \dots (H_L)$ denote the isotropy types of $M$, and  arrange them in such a way that 
\bqn
H_j \text{ is conjugated to a subgroup of }H_i  \quad \Rightarrow \quad i \leq j.
\eqn
Let $H\subset K$ be a closed subgroup, and $M(H)$ the union of all orbits of type $K/H$. Then $M$ has a stratification into orbit types according to  
\bqn
M=M(H_1) \cup \dots \cup M(H_L).
\eqn
By the principal orbit theorem, the set $M(H_L)$ is open and dense in $M$, while $M(H_1)$ is a closed, $K$-invariant submanifold. Denote by $\nu_1$ the normal $K$-vector bundle of $M(H_1)$, and by $f_1: \nu_1 \rightarrow M$ a $K$-invariant tubular neighbourhood of $M(H_1)$ in $M$. Take a $K$-invariant metric on $\nu_1$, and put
\bqn
{D}_t(\nu_1)=\mklm {v \in \nu_1: \norm{v} \leq t }, \qquad t >0.
\eqn
We then define the compact, $K$-invariant submanifold with boundary
\bqn
M_2=M - f_1(\stackrel{\circ}{D}_{1/2}(\nu_1)), 
\eqn
on which the isotropy type $(H_1)$ no longer occurs, and endow it with a $K$-invariant Riemannian metric with product form in a $K$-invariant collar neighborhood of $\gd M_2$ in $M_2$. Consider now the union $M_2(H_2)$ of orbits in $M_2$ of type $K/H_2$, a compact $K$-invariant submanifold of $M_2$ with boundary, and let $f_2:\nu_2 \rightarrow M_2$ be a $K$-invariant tubular neighbourhood  of $M_2(H_2)$ in $M_2$, which exists due to the particular form of the metric on $M_2$. Taking a $K$-invariant metric on $\nu_2$, we define
\bqn
M_3=M_2 - f_2(\stackrel{\circ}{D}_{1/2}(\nu_2)), 
\eqn
which constitutes a compact $K$-invariant submanifold with corners and isotropy types $(H_3), \dots (H_L)$. Continuing this way, one finally obtains for $M$ the decomposition 
\bqn  
M= f_1({D}_{1/2}(\nu_1)) \cup \dots  \cup f_L({D}_{1/2}(\nu_L)),
\eqn
where we identified  $f_L({D}_{1/2}(\nu_L))$ with $M_L$, which leads to the covering 
\bqn
M= f_1(\stackrel{\circ}{D}_{1}(\nu_1)) \cup \dots \cup f_L(\stackrel{\circ}{D}_{1}(\nu_L)),\qquad  f_L(\stackrel{\circ}{D}_{1}(\nu_L))=\stackrel{\circ} M_L.
\eqn
 In exactly the same way, one shows the existence of a covering 
\bqn
\label{covrn}
\rn= f_1(\stackrel{\circ}{D}_{1}(\nu_1)) \cup \dots \cup  f_L(\stackrel{\circ}{D}_{1}(\nu_L))
\eqn
of $\rn$ by $G$-invariant tubular neighbourhoods, where  $ f_L(\stackrel{\circ}{D}_{1}(\nu_L))\equiv \stackrel{\circ} \rn_L$, the notation being as before. 

\section{The desingularization process}

Let us now start resolving the singularities of the critical set $\Crit(\psi)$.  For this, we will have to set up an iterative desingularization process along the strata of the underlying $G$-action, where each step in our iteration will consist of a decomposition, a monoidal transformation, and a reduction. For simplicity, we shall assume that at each iteration step the set of maximally singular orbits is connected. Otherwise each of the connected components, which might even have different dimensions,  has to be treated separately. 

\subsection*{First decomposition} As before, let $f_k:\nu_k\rightarrow M_k$ be an invariant tubular neighborhood of $M_k(H_k)$ in 
\bdm
M_k=\rn-\bigcup_{i=1}^{k-1} f_i(\stackrel{\circ}{D}_{1/2}(\nu_i)),
\edm 
a manifold with corners on which $G$ acts with the isotropy types $(H_k), (H_{k+1}), \dots, (H_L)$, and put  $W_k=f_k(\stackrel{\circ}{D_1}(\nu_k))$. Introduce a partion of unity $\mklm{\chi_k}_{k=1,\dots,L}$ subordinated to the covering $\mklm{W_k}$, and define 
\bq
\label{intj}
I_k(\mu)=   \int _{T^\ast \rn}  \int_{\g} e^{i \psi( x,\xi,X)/\mu }   a(x,\xi,X)   \chi_k(x) \d X \d\xi \d x ,   
\eq
so that $I(\mu)=I_1(\mu)+\dots +I_L(\mu)$. Now, if $(x,\xi) \in \Omega$, and either $x$ or $\xi$ belong to $\rn(H_L)$,  by Cassanas-Ramacher \cite{cassanas-ramacher}, Proposition 2, it follows already that $(x,\xi) \in \mathrm{Reg}\, \Omega$. The critical set of $\psi$ is therefore a smooth manifold in a neighborhood of $ \supp \chi_L a$, since $ f_L(\stackrel{\circ}{D}_{1}(\nu_L)) \subset \rn(H_L)$. Furthermore, it is clear that the transversal Hessian of $\psi$ is non-degenerate on $\mathrm{Reg}\,  \Crit(\psi)$. For this reason, the stationary phase theorem can directly be applied to compute the integral $I_L(\mu)$.  Let us therefore turn to the case that $k \in \mklm{1, \dots, L-1}$. The sets 
\begin{align*}
  \Omega_k&=\mklm{(x,\xi) \in {W_k}\times\rn: \eklm {Ax,\xi}=0 \text{ for all } A \in \g}, \\ \Crit_k(\psi) &=\mklm{(x,\xi,X)\in \Omega_k \times \g: X\in \g_{(x,\xi)} }
\end{align*}
are then no longer smooth manifolds, and since $\supp \chi_k \subset W_k$, the stationary phase theorem can not be applied directly in this situation. Instead, we shall resolve the singularities of $\Crit_k(\psi)$, and after this  apply the principle of the stationary phase in a suitable resolution space. For this, introduce for each $p^{(k)}\in M_k(H_k)$ the decomposition
\bqn
\g=\g_{p^{(k)}}\oplus \g_{p^{(k)}}^\perp, 
\eqn
where $\g_{p^{(k)}}$ denotes the Lie algebra of stabilizer $G_{p^{(k)}}$ of $p^{(k)}$, and $\g_{p^{(k)}}^\perp$ its orthogonal complement with respect to the scalar product $\tr (^tAB)$ in $\g$. Let further $A_1(p^{(k)}), \dots, A_{d^{(k)}}(p^{(k)})$ be an orthonormal basis of $\g_{p^{(k)}}^\perp$, and $B_1(p^{(k)}),\dots, B_{e^{(k)}}(p^{(k)})$ an orthonormal basis of $\g_{p^{(k)}}$. Consider  the isotropy algebra bundle over $M_k(H_k)$
\bqn
\mathfrak{iso} \,M_k(H_k) \rightarrow M_k(H_k),
\eqn
as well as the canonical projection
\bqn 
\pi_k: W_k \rightarrow M_k(H_k), \qquad f_k(p^{(k)},v^{(k)}) \mapsto p^{(k)}, \qquad p^{(k)} \in M_k(H_k), \, v^{(k)} \in (\nu_k)_{p^{(k)}},
\eqn
where $f_k(p^{(k)},v^{(k)})=(\exp_{p^{(k)}} \circ \gamma ^{(k)})( v^{(k)})$, and $\gamma^{(k)}$ is some scaling function, see \cite{bredon}, page 306-307. We then consider the induced bundle
\bqn
\pi_k^\ast \frak{iso}\,  M_k(H_k)=\mklm {(f_k(p^{(k)},v^{(k)}),X)\in W_k \times \g: X \in \g_{p^{(k)}}},
\eqn
and denote by  
$$\Pi_k: W_k \times \g \rightarrow  \pi_k^\ast \frak{iso} \, M_k(H_k)$$
the canonical projection which is obtained by considering geodesic normal coordinates around $\pi_k^\ast \, \frak{iso} M_k(H_k)$, and  identifying  $W_k\times \g$ with a neighborhood of the zero section in  the normal bundle $N\, \pi_k^\ast \,\frak{iso} \, M_k(H_k)$. Note also that 
 the fiber of the normal bundle to $\pi^\ast \frak{iso}\,  M_k(H_k)$ at a point $(f_k(p^{(k)},v^{(k)}),X)$ can be identified with $\g_{p^{(k)}}^\perp$. Integrating along the fibers of the normal bundle to  $\pi_k^\ast \, \frak{iso} M_k(H_k)$ we therefore obtain  for $I_k(\mu)$ the expression
\begin{align*}
 I_k(\mu)&=\int_{\pi_k^\ast \, \frak{iso} M_k(H_k)} \left [\int_{\Pi_k^{-1}(p^{(k)},v^{(k)},B^{(k)})\times \rn
} e^{i\psi/\mu}\chi_k a \, \Phi_k \, \d \xi \, dA^{(k)}   \right ]  dB^{(k)} \,  dv^{(k)}  dp^{(k)} \\
 &=\int_{ M_k(H_k) }\left [\int_{ \g \times \pi_k^{-1}(p^{(k)})\times\rn
} e^{i\psi/\mu}\chi_k a \, \Phi_k \, \d \xi \, dA^{(k)} \, dB^{(k)} \, dv^{(k)}  \right ]      dp^{(k)},
\end{align*}
where 
\bqn
\gamma^{(k)} \big (\stackrel \circ D_1(\nu_k)_{p^{(k)}}\big )  \times  \g_{p^{(k)}}^\perp \times   \g_{p^{(k)}} \ni (v^{(k)},  A^{(k)},B^{(k)})\mapsto (\exp_{p^{(k)}} v^{(k)},A^{(k)}+B^{(k)})=(x,X)
\eqn
are coordinates on $\g \times \pi_k^{-1}(p^{(k)})$, while $dp^{(k)}$, and $dA^{(k)}, dB^{(k)} ,  dv^{(k)} $ are suitable measures in   $M_k(H_k)$, and  $\g_{p^{(k)}}^\perp$, $\g_{p^{(k)}}$, $\stackrel \circ D_1(\nu_k)_{p^{(k)}}$, respectively, such that $ \d X \d x\equiv\Phi_k \,  dA^{(k)} \, dB^{(k)} \,  dv^{(k)} \, dp^{(k)}$. 

\subsection*{First monoidal transformation}  Let now $k \in \mklm{1, \dots, L-1} $ be fixed. For the further analysis of the integral $I_k(\mu)$, we shall sucessively resolve the singularities of $\Crit_k(\psi)$, until we are in position to apply the principle of the stationary phase in a suitable resolution space. To begin with, we perform a monoidal transformation 
\bqn 
\zeta_k: B_{Z_k}( W_k \times \g) \longrightarrow W_k \times \g
\eqn
 in $W_k \times \g$ with center $Z_k= \frak{iso} \, M_k(H_k)$. For this, let us write  $ A^{(k)}(p^{(k)},\alpha^{(k)})=\sum  \alpha_i^{(k)} A_i^{(k)}(p^{(k)})$,  $ B^{(k)}(p^{(k)},\beta^{(k)})=\sum  \beta_i^{(k)} B_i^{(k)}(p^{(k)})$, and 
\bqn
 v^{(k)}(p^{(k)},\theta^{(k)})= \sum _{i=1}^{c^{(k)}} \theta_i^{(k)} v_i^{(k)}(p^{(k)}),
\eqn
where $\{v_1^{(k)}(p^{(k)}),\dots  ,v_{c^{(k)}}^{(k)}(p^{(k)})  \}$ denotes an orthonormal frame in $\nu_k$. With respect to these coordinates we have $Z_k=\mklm{\alpha^{(k)}=0, \, \theta^{(k)}=0}$, so that 
\begin{gather*}
B_{Z_k}( W_k \times \g)=\mklm{ (x,X,[t]) \in W_k \times \g \times \mathbb{RP}^{c^{(k)} + d^{(k)}-1}:  \theta^{(k)}_i t_j = \theta^{(k)}_j t_i, \, \alpha^{(k)}_i t_{c^{(k)}+j}=\alpha^{(k)}_j t_{c^{(k)}+i}  },\\
\zeta_k: (x,X,[t]) \longmapsto (x,X).
\end{gather*}
If we now cover $B_{Z_k}( W_k \times \g)$ with the charts $\mklm{(\phi_\rho, U_\rho)}$,  $U_\rho=B_{Z_k}( W_k \times \g)\cap (W_k \times \g \times V_\rho)$, where $V_\rho=\mklm{[t] \in  \mathbb{RP}^{c^{(k)} + d^{(k)}-1}: t_\rho\not=0}$, we obtain for $\zeta_k$ in each of the $\theta^{(k)}$-charts $\mklm{U_\rho}_{1\leq \rho\leq c^{(k)}}$ the expressions
\bqn
\, ^\rho \zeta_k=\zeta_k \circ \phi_\rho: ( p^{(k)},\tau_k,\, ^\rho \tilde v^{(k)},  A^{(k)}, B^{(k)}) \mapsto (\exp_{p^{(k)}}   \tau_k  \,^\rho\tilde v^{(k)}, \tau_k A^{(k)}+ B^{(k)})\equiv(x,X),
\eqn
where 
\bqn
 \,^\rho \tilde v^{(k)}(p^{(k)},\theta^{(k)})= \Big (v_\rho^{(k)}(p^{(k)})+ \sum _{i\not=\rho}^{c^{(k)}} \theta_i^{(k)} v_i^{(k)}(p^{(k)})\Big ) \Big / \sqrt{1 + \sum_{i\not=\rho} (\theta_i^{(k)} )^2} \in (\,^\rho S_k^+)_{p^{(k)}},
\eqn
and 
$$\,^\rho S_k^+=\mklm{v \in \nu_k: v = \sum s_i v_i, s_\rho>0,  \norm{v}=1}.$$
  Note that for each $1 \leq \rho\leq c^{(k)}$, $$W_k \simeq f_k(\,^\rho S_k^+ \times (-1,1))$$ up to a set of measure zero, and since 
  $f_k(p^{(k)}, v^{(k)})=(\exp_{p^{(k)}} \circ \gamma^{(k)})(v^{(k)})$, we have $\tau_k \in (-T,T)$ for some $1>T>0$.  As a consequence, we obtain for the phase function the factorization
\begin{align*}
\psi(x,\xi,X)&=\eklm{\big (\tau_k A^{(k)}+ B^{(k)}\big )\exp_{p^{(k)}} \tau_k\, ^\rho \tilde v^{(k)},\xi}\\ & =\tau_k \left [ \eklm{A^{(k)} p^{(k)} + B^{(k)} \,^\rho \tilde v^{(k)}, \xi} +\tau_k \eklm {A^{(k)}  \,^\rho \tilde v^{(k)},\xi }\right ],
\end{align*} 
where we took into account that $\exp_{p^{(k)}} v^{(k)} = p^{(k)}+v^{(k)}$. Similar considerations hold for $\zeta_k$ in the $\alpha^{(k)}$-charts $\mklm{U_\rho}_{c^{(k)}+1 \leq \rho \leq c^{(k)}+d^{(k)}}$, so that we get
\bqn 
 \psi \circ (\id_\xi \otimes \zeta_k) =  \,^{(k)} \tilde \psi^{tot}=\tau_k \cdot  \,  ^{(k)} \phw, 
 \eqn
$ \,^{(k)} \tilde \psi^{tot}$ and $  \,  ^{(k)} \phw $ being the \emph{total} and \emph{weak transform} of the phase function $\psi$, respectively.\footnote{For an explanation of this notation, see section \ref{sec:6}.} Introducing a partition $\mklm{u_\rho}$ of unity subordinated to the covering $\mklm{U_\rho}$ now yields 
\bqn 
I_k(\mu)=\sum_{\rho=1} ^{c^{(k)}} \, ^\rho I_k(\mu)+\sum_{\rho=c^{(k)}+1} ^{d^{(k)}} \, ^\rho \tilde I_k(\mu),
\eqn 
where the integrals  $ ^\rho I_k(\mu)$ and $ ^\rho \tilde I_k(\mu)$ are given by the expressions
\begin{gather*}
\int_{ M_k(H_k) }\left [\int_{\, ^\rho \zeta^{-1}_k( \g \times \pi_k^{-1}(p^{(k)}))\times\rn
}  (u_\rho \circ \phi_\rho) \, (\,^\rho\zeta_k )^\ast (  e^{i\psi /\mu }\chi_k a \,  \Phi_k  \,  d \xi  \, dA^{(k)} \, dB^{(k)} \, d v^{(k)} )  \d \xi \,  \right ]      dp^{(k)}.
\end{gather*}
As we shall see in Section \ref{sec:8}, the weak transform $\, ^{(k)} \phw$ has no critical points in the $\alpha^{(k)}$-charts, which will imply that the integrals $ ^\rho \tilde I_k(\mu)$ contribute to $I(\mu)$ only with higher order terms. In what follows, we shall therefore restrict ourselves to the situation where $\chi_k a   \circ ( \id_\xi \otimes \zeta_k)$ has compact support in one of the $\theta^{(k)}$-charts. Thus we can assume $I_k(\mu)$ to be given by  
\begin{gather*}
\int_{ M_k(H_k) }\left [\int_{\zeta^{-1}_k( \g \times \pi_k^{-1}(p^{(k)}))\times\rn
} e^{i\frac{\tau_k}\mu  \,  ^{(k)} \phw}(\chi_k a\circ (\id_\xi \otimes \zeta_k) ) \, \tilde \Phi_k \, \d \xi \, dA^{(k)} \, dB^{(k)} \, d\tilde v^{(k)} \d \tau_k \right ]      dp^{(k)} \\
=\int_{ M_k(H_k) \times (-T,T)  }\Big [\int_{ (S_k^+)_{p^{(k)}} \times  \g_{p^{(k)}} \times \g_{p^{(k)}}^\perp \times\rn
} e^{i\frac{\tau_k}\mu \,  ^{(k)} \phw}(\chi_k a\circ (\id_\xi \otimes \zeta_k) ) \, \tilde \Phi_k \\   \d \xi \, dA^{(k)} \, dB^{(k)} \, d\tilde v^{(k)}   \Big ]  \d \tau_k \,     \d p^{(k)},
\end{gather*}
 where we skipped the index $\rho$, and  took into account that 
\bqn 
\zeta^{-1}_k( \g \times \pi_k^{-1}(p^{(k)}))=\{p^{(k)}\} \times (-T,T) \times (S_k^+)_{p^{(k)}} \times  \g_{p^{(k)}} \times \g_{p^{(k)}}^\perp.
\eqn
Here $d\tilde v^{(k)}$ is a suitable measure on $(S_k^+)_{p^{(k)}}$ such that $\d X \d x \equiv \tilde \Phi_k \, d A^{(k)} \, dB^{(k)} \,  d\tilde v^{(k)} \, d\tau_k \, dp^{(k)}$. Furthermore, a computation shows that  
\bqn 
\tilde \Phi_k  = |\tau_k|^{c^{(k)}+d^{(k)}-1} \, \Phi_k\circ \zeta_k. 
\eqn

\subsection*{First reduction} Let us now assume that there exists a $x \in W_k$ with orbit type $G/H_j$, and let  $p^{(k)} \in M_k(H_k), v^{(k)} \in (\nu_k)_{p^{(k)}}$ be such that $x=f_k(p^{(k)},v^{(k)})$. Since x lies in a slice at $p$ around the $G$-orbit of $p^{(k)}$, we have $G_x \subset G_{p^{(k)}}$ by Bredon \cite{bredon}, page 86. Hence, $H_j\simeq G_x$ must be conjugated to a subgroup of $H_k\simeq G_{p^{(k)}}$. Now, $G$ acts on $M_k$ with the isotropy types $(H_k),(H_{k+1}), \dots, (H_L)$. The isotropy types occuring in $W_k$ are therefore those for which the corresponding isotropy groups  $H_k,H_{k+1}, \dots, H_L$ are conjugated to a subgroup of $H_k$, and we shall denote them by
\bqn
(H_k) = (H_{i_1}), (H_{i_2}), \dots, (H_L).
\eqn
Consequently, $G$ acts on $S_k$ with the isotropy types $(H_{i_2}), \dots, (H_L)$. Now, for every $p^{(k)}\in M_k(H_k)$, $ (\nu_k)_{p^{(k)}}$ is an orthogonal $ G_{p^{(k)}}$-space; furthermore,  by the invariant tubular neighborhood theorem, one has the isomorphism
\bqn
W_k/G \simeq (\nu_k)_{p^{(k)}}/ G_{p^{(k)}}.
\eqn
Therefore $G_{p^{(k)}}$ acts on the manifold $(S_k)_{p^{(k)}}$ with the  isotropy types $(H_{i_2}), \dots, (H_L)$ as well.
As will turn out, if $G$ acted on $S_k$ only with type $(H_L)$, the critical set of $^{(k)} \phw$ would be clean in the sense of Bott, and we could proceed to apply the stationary phase theorem to compute $I_k(\mu)$. But in general this will not be the case, and we are forced to continue with the iteration.

\subsection*{Second decomposition}
For every fixed  $p^{(k)}\in M_k(H_k)$ consider the covering of the compact $G_{p^{(k)}}$-manifold  $(S_k)_{p^{(k)}}$ given by 
\bqn
(S_k)_{p^{(k)}} =W_{ki_2} \cup \dots \cup W_{kL}, \qquad W_{ki_j}= f_{ki_j}(\stackrel \circ D_1(\nu_{ki_j})), \quad W_{kL}=\mathrm{Int} ( (S_k)_{p^{(k)},L}),
\eqn
where $f_{ki_j}:\nu_{ki_j} \rightarrow (S_k)_{p^{(k)},i_j}$ is an invariant tubular neighborhood of $(S_k)_{p^{(k)},i_j}(H_{i_j})$ in 
\bqn 
(S_k)_{p^{(k)},i_j}=(S_k)_{p^{(k)}} - \bigcup_{r=2}^{j-1} f_{ki_r}(\stackrel \circ D_{1/2}(\nu_{ki_r})),  \qquad j\geq 2, 
\eqn
and $f_{ki_j}(p^{(i_j)},v^{(i_j)})=(\exp_{p^{(i_j)}} \circ \gamma^{(i_j)})(v^{(i_j)})$, $p^{(i_j)} \in  (S_k)_{p^{(k)},i_j}(H_{i_j})$, $ v ^{(i_j)} \in ( \nu_{ki_j}) _{p^{(i_j)}}$. 
Let further $\{\chi_{ki_j}\}$ denote a partition of the unity subordinated to the covering $\mklm{W_{ki_j}}$,  and define
\begin{align*}
I_{ki_j}(\mu) =&\int_{ M_k(H_k)\times (-T,T) }\Big [\int_{ (S_k^+)_{p^{(k)}} \times  \g_{p^{(k)}} \times \g_{p^{(k)}}^\perp \times\rn
} e^{i\frac{\tau_k}\mu \,  ^{(k)} \phw}(\chi_k a\circ (\id_\xi \otimes \zeta_k) ) \, \chi_{ki_j} \, \tilde \Phi_k \\ &  \d \xi \, dA^{(k)} \, dB^{(k)} \, d\tilde v^{(k)}  \Big ]  \d \tau_k \d p^{(k)},
\end{align*}
so that $I_k(\mu)= I_{ki_2}(\mu) + \dots + I_{kL}(\mu)$. It is important to note that the partition functions $\chi_{ki_j}$ depend smoothly on $p^{(k)}$ as a consequence of the tubular neighborhood theorem, by which in particular $S_k /G \simeq (S_k)_{p^{(k)}}/G_{p^{(k)}}$, and the smooth dependence in $p^{(k)}$ of the Riemannian metrics on the normal bundles $\nu_{ki_j}$ and the manifolds with corners $(S_k)_{p^{(k)}, i_j}$.
Since $G_{p^{(k)}}$ acts on $W_{kL}$ only with type $(H_L)$, the iteration process for $I_{kL}(\mu)$ ends here. For the remaining integrals $I_{ki_j}(\mu)$ with $k < i_j < L$, let us denote by
\bqn
\frak{iso} \,(S_k)_{p^{(k)},i_j}(H_{i_j}) \rightarrow (S_k)_{p^{(k)},i_j}(H_{i_j})
\eqn
the isotropy algebra bundle over $(S_k)_{p^{(k)},i_j}(H_{i_j})$, and by $\pi_{ki_j}: W_{ki_j} \rightarrow (S_k)_{p^{(k)},i_j}(H_{i_j})$ the canonical projection.  For $p^{(i_j)} \in (S_k)_{p^{(k)},i_j}(H_{i_j})$, consider the decomposition
\bqn
\g = \g_{p^{(k)}} \oplus \g_{p^{(k)}}^\perp =(\g_{p^{(i_j)}}\oplus \g_{p^{(i_j)}}^\perp) \oplus \g_{p^{(k)}}^\perp.
\eqn
Let further $A_1^{(i_j)}, \dots ,A_{d^{(i_j)}}^{(i_j)}$ be an orthonormal frame in $ \g_{p^{(i_j)}}^\perp$,  as well as $B_1^{(i_j)}, \dots ,B_{e^{(i_j)}}^{(i_j)}$ be an orthonormal frame in $ \g_{p^{(i_j)}}$, and $v_1^{(ki_j)}, \dots, v_{c^{(i_j)}}^{(ki_j)}$ an orthonormal frame in $(\nu_{ki_j}) _{p^{(i_j)}}$. Integrating along the fibers in a neighborhood of $\pi_{ki_j}^\ast \frak{iso} \, (S_k)_{p^{(k)},i_j}(H_{i_j}) \subset W_{ki_j} \times \g_{p^{(k)}}$ then yields  for $ I_{ki_j}(\mu)$ the expression
\begin{align*}
I_{ki_j}(\mu) &=  \int_{M_k(H_k)\times (-T,T)} \Big [ \int_{(S_k)_{p^{(k)},i_j}(H_{i_j})} \Big [ \int_{\pi_{ki_j}^{-1} (p^{(i_j)})\times \g_{p^{(k)}}\times \g_{p^{(k)}}^\perp \times  \rn} e^{i\frac{\tau_k}\mu \, ^{(k)} \phw }  \\ & \times  (\chi_k a  \circ (\id_\xi \otimes \zeta_k) ) \chi_{ki_j} \,   \Phi_{ki_j} \, \d \xi  \, d A^{(k)} \, d A^{(i_j)} \, dB^{(i_j)} \,  dv^{(i_j)}     \big ] dp^{(i_j)}  \Big ]   d\tau_k dp^{(k)},
\end{align*}
where $\Phi_{ki_j}$ is a Jacobian, and 
\bqn
\gamma^{(i_j)}\big (  \stackrel \circ D_1(\nu_{ki_j})_{p^{(i_j)}}\big )  \times \g_{p^{(i_j)}}^\perp \times   \g_{p^{(i_j)}} \ni (v^{(i_j)}, A^{(i_j)},B^{(i_j)})\mapsto (\exp_{p^{(i_j)}} v^{(i_j)},A^{(i_j)} + B^{(i_j)})=(\tilde v^{(k)},B^{(k)})
\eqn
are coordinates on $ \pi_{ki_j}^{-1}(p^{(i_j)})\times \g_{p^{(k)}}$, while $dp^{(i_j)}$, and $dA^{(i_j)}, dB^{(i_j)} ,  dv^{(i_j)} $ are suitable measures in the spaces   $ (S_k)_{p^{(k)},i_j}(H_{i_j})$, and  $\g_{p^{(i_j)}}^\perp$, $\g_{p^{(i_j)}}$, $\stackrel \circ D_1(\nu_{ki_j})_{p^{(i_j)}}$, respectively, such that we have the equality $ \tilde \Phi_k \d B^{(k)} \d \tilde v^{(k)}\equiv\Phi_{ki_j} \,  dA^{(i_j)} \, dB^{(i_j)} \,  dv^{(i_j)} \, dp^{(i_j)}$.

\subsection*{Second monoidal transformation}

Let us fix an $l$ such that $k< l < L$, and consider in the $\theta^{(k)}$-chart $(-T,T)\times S_k^+ \times \g$ a monoidal transformation 
\bqn 
\zeta_{kl}: B_{Z_{kl}}((-T,T)\times S_k^+ \times \g) \longrightarrow (-T,T)\times S_k^+ \times \g
\eqn
with center
\bqn
Z_{kl}=  (-T,T)\times \frak{iso} \,S_{k,l}^+(H_l), \qquad S_{k,l} = \bigcup _{p^{(k)} \in M_k(H_k)} (S_k)_{p^{(k)},l}.
\eqn
Writing  $ A^{(l)}(p^{(l)},\alpha^{(l)})=\sum  \alpha_i^{(l)} A_i^{(l)}(p^{(l)})$,  $ B^{(l)}(p^{(l)},\beta^{(l)})=\sum  \beta_i^{(l)} B_i^{(l)}(p^{(l)})$, and 
\bqn
 v^{(l)}(p^{(l)},\theta^{(l)})= \sum _{i=1}^{c^{(l)}} \theta_i^{(l)} v_i^{(kl)}(p^{(l)}),
\eqn
one has $Z_{kl}=\mklm{\alpha^{(k)}=0, \, \alpha^{(l)}=0, \, \theta^{(l)}=0}$. If we now cover $B_{Z_{kl}}((-T,T)\times S_k^+ \times \g)$ with the standard charts, we shall see again in Section \ref{sec:8} that   modulo higher order terms we can assume that $((\chi_k a   \circ ( \id_\xi \otimes \zeta_k))\chi_{kl}) \circ \zeta_{kl}$ has compact support in one of the $\theta^{(l)}$-charts. Therefore it suffices to examine $\zeta_{kl}$ in one of these charts, in which it reads
\begin{gather*}
\zeta_{kl}: (p^{(k)},\tau_k, p^{(l)}, \tau_l, \tilde v^{(l)}, A^{(k)},  A^{(l)},  B^{(l)}) \mapsto \\\mapsto (p^{(k)},\tau_k, \exp_{p^{(l)}} \tau_l \tilde v^{(l)}, \tau_l A^{(k)}, \tau_l  A^{(l)}+ B^{(l)})\equiv(p^{(k)},\tau_k, \tilde v^{(k)},A^{(k)}, B^{(k)}),
\end{gather*}
where
\bqn
\tilde v^{(l)}(p^{(l)},\theta^{(l)})= \Big (v_\sigma^{(kl)}(p^{(l)})+ \sum _{i\not=\sigma}^{c^{(l)}} \theta_i^{(l)} v_i^{(kl)}(p^{(l)})\Big) \Big / \sqrt{1 +  \sum _{i\not=\sigma} (\theta_i^{(l)})^2}
\eqn
for some $\sigma$. Note that $Z_{kl}$ has normal crossings with the exceptional divisor $E_k=\zeta_k^{-1}(Z_k) = \mklm{\tau_k=0}$, and that 
$$W_{kl} \simeq f_{kl} (S_{kl}^+ \times (-1,1))$$ up to a set of measure zero, where $S_{kl}$ denotes  \footnote{In order not to overload notation, we have denoted by $S_{kl}$ and $S_{k,l}$ two quite different sets.}  the sphere subbundle in $\nu_{kl}$,  and we set $S_{kl}^+=\mklm { v \in S_{kl}:  v=\sum v_i v_i^{(kl)}, \, v_\sigma>0 }$. Taking into account that $\exp_{p^{(l)}}\tau_l \tilde v^{(l)}= (\cos \tau_l) \, p^{(l)} +(\sin\tau_l)\, \tilde v^{(l)}$, one sees that the phase function factorizes according to 
\bqn 
\psi \circ (\id_\xi \otimes (\zeta_k \circ \zeta_{kl}))= \,^{(kl)} \tilde \psi^{tot}=\tau_k \, \tau_l \cdot \,  ^{(kl)} \phw,
\eqn
which in the given charts reads
\begin{align*}
\psi(x,\xi,X)&=\tau_k \left [ \eklm{ \tau_l A^{(k)} \, p ^{(k)}+ (\tau_l A^{(l)}+B^{(l)}) \exp_{p^{(l)}} \tau_l \tilde v^{(l)}, \xi}+\tau_k\tau_l \eklm{ A^{(k)} \tilde v ^{(k)} ,\xi} \right ] \\
&=\tau_k \tau_l \left [ \eklm{A^{(k)} \, p ^{(k)} +  A^{(l)} \, p ^{(l)}+B^{(l)} \tilde v^{(l)}, \xi}+O(|\tau_k \,A^{(k)}|)+O(|\tau_l \,A^{(l)}|)+O(|\tau_l B^{(l)}\tilde v^{(l)}|) \right ],
\end{align*}
where 
\begin{align*} 
O(|\tau_k \,A^{(k)}|)&=\tau_k \eklm{ A^{(k)} \tilde v ^{(k)} ,\xi}, \\
O(|\tau_l \,A^{(l)}|)&=\eklm{(\cos \tau_l -1) A^{(l)} p^{(l)} + \sin \tau_l A^{(l)}  \tilde v^{(l)} , \xi },\\
O(|\tau_l B^{(l)}\tilde v^{(l)}|)&=\eklm{(\tau_l^{-1} \sin \tau_l -1) B^{(l)} \tilde v^{(l)} , \xi }. 
\end{align*}
Since 
\begin{gather*}
\zeta_{kl}^{-1}(\{p^{(k)}\} \times \{ \tau_k\} \times \pi_{kl}^{-1} (p^{(l)})\times \g_{p^{(k)}}\times \g_{p^{(k)}}^\perp )\\=\{p^{(k)}\} \times \{\tau_k\} \times \{p^{(l)}\} \times (-T,T) \times (S_{kl}^+)_{p^{(l)}} \times \g _{p^{(l)}}\times \g _{p^{(l)}}^\perp \times \g _{p^{(k)}}^\perp,
\end{gather*}
we obtain for $ I_{kl}(\mu)$ the expression
\begin{align*}
I_{kl}(\mu) &=\int_{M_k(H_k)\times (-T,T)} \Big [ \int_{(S_k)_{p^{(k)},l}(H_{l})} \Big [ \int_{\zeta_{kl}^{-1}(\{p^{(k)}\}\times \{ \tau_k\} \times \pi_{kl}^{-1} (p^{(l)})\times \g_{p^{(k)}}\times \g_{p^{(k)}}^\perp )\times \rn} e^{i\frac{\tau_k\tau_l }\mu  \, ^{(kl)} \phw}  \\ &\times  ( (\chi_k a  \circ (\id_\xi \otimes \zeta_k )) \chi_{kl} ) \circ \zeta_{kl} \,  \tilde  \Phi_{kl} \, \d \xi  \, d A^{(k)} \, d A^{(l)} \, dB^{(l)} \,  d\tilde v^{(l)}  \, d\tau_l \Big ] dp^{(l)}  \Big ] \, d\tau_k \,   dp^{(k)}\\
&=\int_{M_k(H_k)\times (-T,T)} \Big [ \int_{(S_k)_{p^{(k)},l}(H_{l})\times (-T,T)} \Big [ \int_{(S^+_{kl})_{p^{(l)}} \times \g_{p^{(l)}} \times \g_{p^{(l)}}^\perp \times \g_{p^{(k)}}^\perp )\times \rn} e^{i\frac{\tau_k\tau_l }\mu  \, ^{(kl)} \phw}  \\ &\times  ( (\chi_k a  \circ (\id_\xi \otimes \zeta_k )) \chi_{kl} ) \circ \zeta_{kl} \,  \tilde  \Phi_{kl} \, \d \xi  \, d A^{(k)} \, d A^{(l)} \, dB^{(l)} \,  d\tilde v^{(l)}   \Big ]  d\tau_l \,dp^{(l)}  \Big ] \, d\tau_k \,   dp^{(k)},
\end{align*}
where $d\tilde v^{(l)}$ is a suitable measure in $(S_{kl}^+)_{p^{(l)}}$ such that we have the equality 
$$\d X \d x \equiv \tilde \Phi_{kl} \,  d A^{(k)} \, dA^{(l)} \, dB^{(l)} \,  d\tilde v^{(l)}  \, d\tau_l \, dp^{(l)}\, d\tau_k \, dp^{(k)}.$$
Furthermore, $\tilde \Phi_{kl} = |\tau_l | ^{c^{(l)} +d^{(k)} +d^{(l)} -1} \Phi_{kl} \circ \zeta_{kl}$.

\subsection*{Second reduction} Now, the group $G_{p^{(k)}}$ acts on $(S_k)_{p^{(k)},l}$ with the isotropy types $(H_l)=(H_{i_j}),(H_{i_{j+1}}), \dots, (H_L)$. By the same arguments given in the first reduction, the isotropy types occuring in $W_{kl}$ constitute a subset of these types, and we shall denote them by
\bqn
(H_l) = (H_{i_{r_1}}), (H_{i_{r_2}}), \dots, (H_L).
\eqn
Consequently, $G_{p^{(k)}}$ acts on $S_{kl}$ with the isotropy types $(H_{i_{r_2}}), \dots, (H_L)$. Again, if $G$ acted on $S_{kl}$ only with type $(H_L)$, we shall see in the next section that the critical set of $^{(kl)} \phw$ would be clean. However, in general this will not be the case, and we have to continue with the iteration.

\subsection*{N-th decomposition}
 The end of the iteration will be reached, once one arrives at a sphere bundle $S_{klmn...}$ on which $G$ acts only with the isotropy type $(H_L)$. More precisely, let  $(H_{i_1}), \dots , (H_{i_{N+1}})=(H_L)$ be a branch of the isotropy tree of the $G$-action in $\rn$, $N\geq 3$,  and consider for every fixed $p^{(i_{N-1})} \in(S_{i_1\dots i_{N-2}})_{p^{(i_{N-2})},i_{N-1}}(H_{i_{N-1}})$ the decomposition of the closed $G_{p^{(i_{N-1})}}$-manifold
 $(S_{i_1\dots i_{N-1}})_{p^{(i_{N-1})}}$ given by 
\begin{gather*}
(S_{i_1\dots i_{N-1}})_{p^{(i_{N-1})}} =W_{i_1\dots i_{N}} \, \cup \, W_{i_1\dots i_{N-1} L}, \\ W_{i_1 \dots i_N}= f_{i_1\dots i_N}(\stackrel \circ D_1(\nu_{i_1\dots i_N})), \quad W_{i_1 \dots i_{N-1}L}=\mathrm{Int}  (S_{i_1\dots i_{N-1}})_{p^{(i_{N-1})},L},
\end{gather*}
where $f_{i_1\dots i_N}:\nu_{i_1\dots i_N } \rightarrow (S_{i_1\dots i_{N-1}})_{p^{(i_{N-1})},i_N}$ is an invariant tubular neighborhood of the closed invariant submanifold $(S_{i_1\dots i_{N-1}})_{p^{(i_{N-1})},i_N}(H_{i_N})$ in $(S_{i_1\dots i_{N-1}})_{p^{(i_{N-1})},i_N}=(S_{i_1\dots i_{N-1}})_{p^{(i_{N-1})}}$, and 
\bqn 
(S_{i_1\dots i_{N-1}})_{p^{(i_{N-1})},L}=(S_{i_1\dots i_{N-1}})_{p^{(i_{N-1})}}-f_{i_1\dots i_N}(\stackrel \circ D_{1/2}(\nu_{i_1\dots i_N})).
\eqn
Let $\mklm{\chi_{i_1\dots i_{N}}, \chi_{i_1\dots i_{N-1} L}}$ denote a partition of unity subordinated to the covering by  the open sets $\{W_{i_1\dots i_{N}}, W_{i_1\dots i_{N-1} L}\}$, and decompose $I_{i_1\dots i_{N-1}}(\mu)$ accordingly, so that
\bqn
I_{i_1\dots i_{N-1}}(\mu)= I_{i_1\dots i_N}(\mu) + I_{i_1\dots i_{N-1}L}(\mu) . 
\eqn
 \subsection*{N-th monoidal transformation} In the chart $(-T,T)^{N-1} \times S_{i_1\dots i_{N-1}}^+ \times \g$ consider the  monoidal transformation 
 \bqn 
 \zeta_{i_1\dots i_N}: B_{Z_{i_1\dots i_N}}((-T,T)^{N-1} \times S_{i_1\dots i_{N-1}}^+ \times \g)\longrightarrow (-T,T)^{N-1} \times S_{i_1\dots i_{N-1}}^+ \times \g
 \eqn
  with center
 \begin{gather*}
 Z_{i_1\dots i_N}= (-T,T)^{N-1} \times \mathfrak{iso} \,  S_{i_1\dots i_{N-1}, i_N}^+(H_{i_N}), \\ S_{i_1\dots i_{N-1}, i_N}=\bigcup _{p^{(i_{N-1})}} (S_{i_1\dots i_{N-1}})_{p^{(i_{N-1})},i_N}=S_{i_1\dots i_{N-1}}.
 \end{gather*}
 The phase function then factorizes according to 
\bqn
\, ^{(i_1\dots i_N)} \tilde \psi^{tot}=\tau_{i_1} \cdots \tau_{i_N} \, ^{(i_1\dots i_N)}\tilde \psi^ {wk},
\eqn
where in the given charts 
\bqn
\, ^{(i_1\dots i_N)}\tilde \psi^ {wk}=\eklm{  \sum _{j=1}^N A^{(i_j)} \, p^{(i_j)}+B^{(i_N)} \tilde v^{(i_N)},\xi} + \sum _{j=1}^N O(|\tau_{i_j} A^{(i_j)}|)+  O(|\tau_{i_N} B^{(i_N)}\tilde v ^{(i_N)}|),
\eqn
and denoting \footnote{Again, note the different meaning of the notations $S_{i_1 \dots i_N}$ and $S_{i_1\dots i_{N-1}, i_N}$.} by $S_{i_1 \dots i_{N}}$ the sphere bundle over $(S_{i_1\dots i_{N-1}})_{p^{(i_{N-1})}, i_{N}} (H_{i_{N}})$, one finally obtains for the integral $I_{i_1\dots i_N}(\mu)$ the expression
\begin{align}
\label{eq:N}
\begin{split}
I&_{i_1\dots i_N}(\mu)=\int_{M_{i_1}(H_{i_1})\times (-T,T)} \Big [ \int_{(S_{i_1})_{p^{(i_1)},i_2}(H_{i_2})\times (-T,T)} \dots \Big [  \int_{(S_{i_1\dots i_{N-1}})_{p^{(i_{N-1})},i_{N}}(H_{i_{N}})\times (-T,T)} \\
&\Big [ \int_{(S_{i_1\dots i_{N}}^+)_{p^{(i_N)}}\times \g_{p^{(i_{N})}}\times \g_{p^{(i_{N})}}^\perp \times \cdots \times \g_{p^{(i_{1})}}^\perp\times \rn} e^{i\frac {\tau_1 \dots \tau_N}\mu \, ^{(i_1\dots i_N)} \tilde \psi ^{wk}}  \, a_{i_1\dots i_N} \,   \tilde \Phi_{i_1\dots i_N} \\
& \d \xi  \d A^{(i_1)} \dots  \d A^{(i_N)}  \d B^{(i_N)} \d \tilde v^{(i_N)} \Big ]  \d \tau_{i_N} \d p^{(i_{N})} \dots  \Big ] \d \tau_{i_2} \d p^{(i_{2})} \Big ]\d \tau_{i_1} \d p^{(i_{1})}.
\end{split}
\end{align}
Here
\bqn 
a_{i_1\dots i_N}=[ a \, \chi_{i_1}  \circ (\id_\xi \otimes \zeta_{i_1} \circ \zeta_{i_1i_2} \circ \dots \circ \zeta_{i_1\dots i_N})] \, [ \chi_{i_1i_2} \circ \zeta_{i_1i_2} \circ \dots\circ \zeta_{i_1\dots i_N}  ] \dots [\chi_{i_1\dots i_N} \circ \zeta_{i_1\dots i_N}]
\eqn
is supposed to have compact support in one of the $\theta^{(i_N)}$-charts, and
\begin{align*}
\tilde \Phi_{i_1\dots i_N} &= |\tau_{i_1}|^{c^{(i_1)}+d^{(i_1)}-1}  |\tau_{i_2}|^{c^{(i_2)}+d^{(i_1)}+d^{(i_2)}-1} \dots |\tau_{i_N}|^{c^{(i_N)}+d^{(i_1)}+\dots + d^{(i_N)}-1}  \Phi_{i_1\dots i_N} \\
&=\prod_{j=1}^N |\tau_{i_j}|^{c^{(i_j)}+\sum_r^j d^{(i_r)}-1}\Phi_{i_1\dots i_N},
\end{align*}
where $\Phi_{i_1\dots i_N}$ is a smooth function which does not depend on the variables $\tau_{i_j}$.

\subsection*{N-th reduction} By assumption, $G$ acts on $S_{i_1\dots i_N}$ only with type $(H_L)$, and the iteration process ends here.

\section{Phase analysis of the weak transform. The first fundamental theorem}

We are now in position to state the first fundamental theorem in  the derivation of equivariant spectral asymptotics. With the notation as in the previous section, consider an iteration of $N$ steps along the branch $((H_{i_1}), \dots , (H_{i_{N+1}})=(H_L))$ of the isotropy tree of the $G$-action in $\rn$, and let 
\begin{gather*}
p^{(i_1)}\in M_{i_1}(H_{i_1}), \quad p^{(i_j)}\in (S_{i_1\dots i_{j-1}}^+)_{p^{(i_{j-1})}, i_j}(H_{i_j}),  \quad j=2,\dots N, \\
\g = \g_{p^{(i_1)}} \oplus \g_{p^{(i_1)}}^\perp =(\g_{p^{(i_2)}}\oplus \g_{p^{(i_2)}}^\perp) \oplus \g_{p^{(i_1)}}^\perp =\dots = \g_{p^{(i_N)}}\oplus \g_{p^{(i_N)}}^\perp \oplus \cdots \oplus \g_{p^{(i_1)}}^\perp\\
d^{(i_j)}=\dim \g_{p^(i_j)}^\perp, \quad e^{(i_j)}=\dim \g_{p^(i_j)}, \quad  j=1,\dots, N.
\end{gather*}
As before, $\mklm{ A_r^{(i_j)}(p^{(i_1)},\dots,p^{(i_j)})}$ will denote a basis of $\g_{p^(i_j)}^\perp$, and $\mklm{ B_r^{(i_N)}(p^{(i_1)},\dots,p^{(i_N)})}$ a basis of $\g_{p^(i_N)}$. Let further 
\begin{align*}
A^{(i_j)} &=\sum_{r=1}^{d^{(i_j)}} \alpha^{(i_j)}_r A_r^{(i_j)}(p^{(i_1)},\dots,p^{(i_j)}), \qquad B^{(i_N)} =\sum_{r=1}^{e^{(i_N)}} \beta^{(i_N)}_r B_r^{(i_N)}(p^{(i_1)},\dots,p^{(i_N)}),
\end{align*}
and put
\bqn
\tilde v^{(i_N)}(p^{(i_j)},\theta^{(i_N)})= \Big (v_\rho^{(i_1\dots i_N)}(p^{(i_j)})+ \sum _{r\not=\rho}^{c^{(i_N)}} \theta_r^{(i_N)} v_r^{(i_1\dots i_N)}(p^{(i_j)})\Big )\Big / \sqrt{1 + \sum\limits_{r\not=\rho} (\theta_r^{(i_N)})^2} 
\eqn
for some $\rho$, where $\mklm{v_r^{(i_1\dots i_N)}(p^{(i_1)}, \dots p^{(i_N)}  )}$ is an orthonormal frame in $(\nu_{i_1\dots  i_N})_{p^{(i_N)}}$.
Finally, we shall use the notations
\begin{align*}
x^{(i_j\dots i_{N})}&=\exp_{p^{(i_j)}}[\tau_{i_j} \exp_{p^(i_{j+1})}[\tau_{i_{j+1}}\exp_{p^(i_{j+2})}[\dots [ \tau_{i_{N-2}}\exp_{p^(i_{N-1})}[\tau_{i_{N-1}}\exp_{p^{(i_N)}}[ \tau_{i_N} \tilde v ^{(i_N)}]]] \dots ]]], \\
X^{(i_j\dots i_{N})}&={\tau_{i_j} \cdots \tau_{i_N}A^{(i_j)}}+{\tau_{i_{j+1}} \cdots \tau_{i_N}A^{(i_{j+1})}}+\dots +{\tau_{i_{N-1}}  \tau_{i_N}A^{(i_{N-1})}} +{\tau_{i_N} A^{(i_N)}}  +B^{(i_N)},
\end{align*}
where $j=1,\dots ,N$. We then have the following

\begin{theorem}
Consider the factorization 
\bqn
\, ^{(i_1\dots i_N)} \tilde \psi^{tot}=\psi (x^{(i_1\dots i_N)}, \xi, X^{(i_1\dots i_N)}) =\tau_{i_1} \cdots \tau_{i_N} \, ^{(i_1\dots i_N)}\tilde \psi^ {wk, \, pre}
\eqn
of the phase function $\psi$ after $N$ iteration steps, where 
\bqn
\, ^{(i_1\dots i_N)}\tilde \psi^ {wk, \, pre}=\eklm{  \sum _{j=1}^N A^{(i_j)} \, p^{(i_j)}+B^{(i_N)} \tilde v^{(i_N)},\xi} + \sum _{j=1}^N O(|\tau_{i_j} A^{(i_j)}|)+  O(|\tau_{i_N} B^{(i_N)}\tilde v ^{(i_N)}|).
\eqn
By construction, for $\tau_{i_j}\not=0$, $1\leq j\leq N$, the $G$-orbit through $x^{(i_1\dots i_N)}$ is of principal type $G/H_L$, which is equivalent to say that $G$ acts on $S_{i_1\dots i_N}$ only with the isotropy type $(H_L)$. Let further 
\bqn 
 \, ^{(i_1\dots i_N)}\tilde \psi^ {wk}
\eqn
denote the pullback of  $ \, ^{(i_1\dots i_N)}\tilde \psi^ {wk,\,pre}$ along the  substitution $\tau=\delta_{i_1\dots i_N}(\sigma)$ given by the sequence of monoidal transformations
\begin{align*}
\delta_{i_1\dots i_N}: (\sigma_{i_1}, \dots \sigma_{i_N}) &\mapsto \sigma_{i_1}( 1, \sigma_{i_2}, \dots, \sigma_{i_N})= (\sigma_{i_1}', \dots ,\sigma_{i_N}')\mapsto \sigma_{i_2}'(\sigma_{i_1}',1,\dots, \sigma_{i_N}')= (\sigma_{i_1}'', \dots, \sigma_{i_N}'')\\
 &\mapsto \sigma_{i_3}''(\sigma_{i_1}'',\sigma_{i_2}'', 1,\dots, \sigma_{i_N}'')= \cdots \mapsto \dots = (\tau_{i_1}, \dots ,\tau_{i_N}).
\end{align*}
Then the critical set $\Crit(\, ^{(i_1\dots i_N)} \phw)$ of $\, ^{(i_1\dots i_N)} \phw$ is given by all points with coordinates
$$(\sigma_{i_1}, \dots, \sigma_{i_N}, p^{(i_1)}, \dots, p^{(i_N)},  \tilde v ^{(i_N)}, \alpha^{(i_1)}, \dots, \alpha^{(i_N)}, \beta^{(i_N)}, \xi) $$
satisfying the conditions

\medskip
\begin{tabular}{ll}
\emph{(I)} &  $\alpha^{(i_j)} =0$ for all $j=1,\dots,N$ , and $\sum \beta_r^{(i_N)} B^{(i_N)}_r\tilde v^{(i_N)}= 0$; \\[2pt]
\emph{(II)} & $\xi \perp \big (\g_{p^{(i_1)}}^\perp \cdot  x^{(i_1\dots i_N)} \big )$;  \, \, $\xi \perp \big (\g_{p^{(i_2)}}^\perp \cdot  x^{(i_2\dots i_N)}\big )$; \dots \, \,  $\xi \perp \big (\g_{p^{(i_N)}}^\perp \cdot  x^{(i_N)}\big )$; \\[2pt]
\emph{(III)} &$\xi \perp \big (\g_{p^{(i_N)}} \cdot  \tilde v^{(i_N)}\big ).$
\end{tabular}
\medskip

\noindent
Furthermore,  $\Crit(\, ^{(i_1\dots i_N)} \phw)$ is a $\Cinft$-submanifold of codimension $2\kappa$, where $\kappa=\dim G/H_L$ is the dimension of a principal orbit.
\end{theorem}
\begin{proof}
Let us compute first the derivatives with respect to $\xi$, and assume that all $\sigma_{i_j}$ are different from zero. Then all $\tau_{i_j}$ are different from zero, too, and $\gd_\xi \, ^{(i_1\dots i_N)} \phw=0$ is equivalent to
\begin{gather*}
\frac 1 {\tau_{i_1} \dots \tau_{i_N}} \gd _\xi \psi (x^{(i_1\dots i_N)}, \xi, X^{(i_1\dots i_N)})=0,
\end{gather*}
which gives us the condition $X^{(i_1\dots i_N)} \in \g_{x^{(i_1\dots i_N)}}$. Since for sufficiently small $\tau_{i_j}$ the point $x^{(i_1\dots i_N)}$ lies in a slice in $N_{p^{(i_1)}} (G \cdot p^{(i_1)})$, the element $X^{(i_1\dots i_N)}$ must annihilate $p^{(i_1)}$ as well. But 
\begin{align*}
\g_{p^{(i_N)}} \subset \g_{p^{(i_{N-1})}} \subset \dots \subset \g_{p^{(i_1)}}
\end{align*}
and $\g_{p^{(i_{j+1})}}^\perp \subset \g_{p^{(i_j)}}$ imply
\bqn
X^{(i_1\dots i_N)}p^{(i_1)}= {\tau_{i_1} \dots \tau_{i_N}\sum \alpha_r^{(i_1)} A_r ^{(i_1)}} p^{(i_1)} =0.
\eqn
Thus we conclude $\alpha^{(i_1)} =0$, which gives $X^{(i_2\dots i_N)} \in \g_{x^{(i_1\dots i_N)}}$, and consequently $X^{(i_2\dots i_N)} \in \g_{x^{(i_2\dots i_N)}}$. Repeating the above argument we actually obtain
\bq
\label{eq:G}
\g_{x^{(i_1 \dots i_N)}} = \g _{\tilde v ^{(i_N)}},
\eq
since $\g_{\tilde v ^{(i_N)}} \subset \g_{p ^{(i_N)}}$, and therefore condition I) in the case that all $\sigma_{i_j}$ are different from zero. Let now one of the  $\sigma_{i_j}$ be equal to zero.  Then all $\tau_{i_j}$ are zero, too, and $\gd_\xi \, ^{(i_1\dots i_N)} \phw=0$ is equivalent to
\bq
\label{eq:B}
 \sum _{j=1}^N \left (\sum_r \alpha^{(i_j)}_r A_r^{(i_j)}\right  ) p^{(i_j)}+\sum_r \beta^{(i_N)}_r B^{(i_N)}_r \tilde v^{(i_N)} =0.
\eq
Now, 
for every $j=1,\dots, N$, the group $G_{p^{(i_j)}}$ acts orthogonally on the space $N_{p^{(i_j)}} ( G_{p^{(i_{j-1})}} \cdot p^{(i_j)})$, where we understand that $G_{p^{(i_0)}}=G$. Furthermore, by construction we have  
\bqn
 N_{p^{(i_{j+1})}} ( G_{p^{(i_{j})}} \cdot p^{(i_{j+1})}) \subset N_{p^{(i_j)}} ( G_{p^{(i_{j-1})}} \cdot p^{(i_j)}), 
\eqn
so that 
\bq
\label{eq:V}
V^{(i_1\dots i_{j})}= \bigcap_{r=1} ^{j} N_{p^{(i_r)}} ( G_{p^{(i_{r-1})}} \cdot p^{(i_r)})= N_{p^{(i_j)}} ( G_{p^{(i_{j-1})}} \cdot p^{(i_j)}).
\eq
Since $p^{(i_j)} \in (S_{i_1\dots i_{j-1}})_{p^{(i_{j-1})}}\subset V^{(i_1\dots i_{j-1})}$, we therefore see that for every $j=2,\dots, N$
\bqn 
\sum_r \alpha_r^{(i_j)} A_r^{(i_j)} \, p^{(i_j)} \in \g _{p^{(i_j)}}^\perp \cdot p^{(i_j)}=  T_{p^{(i_j)}} ( G_{p^{(i_{j-1})}} \cdot p^{(i_j)})\subset V^{(i_1\dots i_{j-1})} .
\eqn 
In addition, one has of course $\sum_r \alpha^{(i_1)} _r A_r^{(i_1)} \, p^{(i_1)} \in \g _{p^{(i_1)}}^\perp \cdot p^{(i_1)}=  T_{p^{(i_1)}} ( G \cdot p^{(i_1)})$, as well as
\bqn
\sum_r \beta^{(i_N)} _r B_r^{(i_N)} \tilde v ^{(i_N)} \in V^{(i_1\dots i_{N})},
\eqn
so that taking everything together we obtain
\bq
\label{eq:I}
\gd_\xi \, ^{(i_1\dots i_N)} \phw=0 \quad \Longleftrightarrow \quad \alpha^{(i_j)}=0 \quad \forall \quad j=1,\dots,N \quad \text{and} \quad \sum_r \beta^{(i_N)} _r B_r^{(i_N)} \tilde v ^{(i_N)} =0.
\eq
Let us consider next the $\alpha$-derivatives. Clearly,  
\begin{align*}
\gd_{\alpha^{(i_1)}} \, ^{(i_1\dots i_N)} \phw=0 \quad & \Longleftrightarrow \quad \eklm{ Y x^{(i_1 \dots i_N)},\xi}=0 \quad \forall \, Y \in \g_{p^{(i_1)}}^\perp,\\ 
\gd_{\alpha^{(i_2)}} \, ^{(i_1\dots i_N)} \phw=0 \quad & \Longleftrightarrow \quad \eklm{ Y x^{(i_2 \dots i_N)},\xi}=0 \quad \forall \, Y \in \g_{p^{(i_2)}}^\perp.
\end{align*}
Now, assuming for a moment that all the $\sigma_{i_j}$ are different from zero,  one computes for the remaining  derivatives that 
\begin{align*}
\gd_{\alpha_r^{(i_j)}} \, ^{(i_1\dots i_N)} \phw &=\frac {1}{\tau_{i_1}\dots  \tau_{i_N}} \eklm{ \tau_{i_j}\dots  \tau_{i_N}A_r^{(i_j)}  x^{(i_1 \dots i_N)},\xi}\\ 
&=\frac 1 {\tau_{i_1}\dots \tau_{i_{j-1}}}  \eklm{ A_r^{(i_j)} (\tau_{i_1}\sin \tau_{i_2} \cdots \sin \tau_{i_{j-1}})x^{(i_j \dots i_N)},\xi},
\end{align*}
since $A_r^{(i_j)} \in \g_{p^{(i_j)}}^\perp \subset \g _{p^{(i_{j-1})}}\subset \dots \subset \g_{p^{(i_1)}}$. From this one deduces for arbitrary $\sigma_{i_j}$ that for $j=1,\dots, N$
\begin{align}
\label{eq:II}
\gd_{\alpha^{(i_j)}} \, ^{(i_1\dots i_N)} \phw=0 \quad & \Longleftrightarrow \quad \eklm{Y  x^{(i_j \dots i_N)},\xi}=0 \quad \forall \, Y \in \g_{p^{(i_j)}}^\perp.
\end{align}
In a similar way, it is not difficult to see that 
\begin{align}
\label{eq:III}
\gd_{\beta^{(i_N)}} \, ^{(i_1\dots i_N)} \phw=0 \quad & \Longleftrightarrow \quad \eklm{ Z \tilde v^{(i_N)},\xi}=0 \quad \forall \, Z \in \g_{p^{(i_N)}},
\end{align}
by which the necessity of the conditions (I)--(III) is established. In order to see their suffficiency, let them be fulfilled, and let us assume again that $\sigma_{i_j}\not=0$ for all $j=1,\dots,N$. Then (II) and (III) imply that 
\bqn 
\eklm{Z\exp_{p^{(i_N)}} \tau_{i_N} \tilde v^{(i_N)}  ,\xi}=0 \qquad \forall Z \in \g_{p^{(i_{N-1})}},
\eqn
since $ \g_{p^{(i_{N-1})}} =  \g_{p^{(i_{N})}}\oplus  \g_{p^{(i_{N})}}^\perp$. By repeatedly using (II) we therefore conclude 
\bq
\label{eq:IVb}
\xi \in N_{x^{(i_1\dots i_N)}}(G\cdot x^{(i_1\dots i_N)}).
\eq
Now, by construction, $G  \cdot x^{(i_1\dots i_N)}$ is of principal type $G/ H_L$ in $\rn$, so that the isotropy group of $x^{(i_1\dots i_N)}$ must act trivially on $N_{x^{(i_1\dots i_N)}}(G\cdot x^{(i_1\dots i_N)})$, compare Bredon \cite{bredon}, page 181. In addition, by (I) and Equation \eqref{eq:G}, $\sum_r \beta^{(i_N)} _r B_r^{(i_N)}  \in \g_{\tilde v ^{(i_N)}} = \g_{x^{(i_1\dots i_N)}}$. The relation \eqref{eq:IVb}  therefore implies $\sum_r \beta^{(i_N)} _r B_r^{(i_N)}  \xi = 0$. Let us consider now the case where at least one of the $\sigma_{i_j}$ equals zero, so that all $\tau_{i_j}=0$. Then (II) means that $\xi \in V^{(i_1\dots i_N)}$. We shall now need  the following simple
\begin{lemma}
The orbit of the point $\tilde v^{(i_N)}$ in the $G_{p^{(i_N)}}$-space $V^{(i_1\dots i_N)}$ is of principal type.
\end{lemma}
\begin{proof}[Proof of the lemma]
By assumption, for  $\sigma_{i_j}\not=0$,  $1 \leq j \leq  N$, the $G$-orbit of $x^{(i_1\dots i_N)}$ is of principal type $G/ H_L$ in $\rn$. The theory of compact group actions then implies that this is equivalent to the fact that $x^{(i_2 \dots i_N)} \in V^{(i_1)}$ is of principal type in the $G_{p^{(i_1)}}$-space $V^{(i_1)}$, see Bredon \cite{bredon}, page 181, which in turn is equivalent to the fact that $x^{(i_3 \dots i_N)} \in V^{(i_1i_2)}$ is of principal type in the $G_{p^{(i_2)}}$-space $V^{(i_1i_2)}$, and so forth. Thus, $x^{(i_j \dots i_N)} \in V^{(i_1 \dots i_{j-1})}$ must be of principal type in the $G_{p^{(i_{j-1})}}$-space $V^{(i_1\dots i_{j-1})}$ for all $j=1,\dots N$, and the assertion follows.
\end{proof}
Now (III) implies that $\xi \in V^{(i_1\dots i_{N})} \cap N_{\tilde v^{(i_N)}} ( G_{p^{(i_N)}} \cdot\tilde v^{(i_N)} )$. But the previous lemma implies that $G_{\tilde v^{(i_N)}}$ acts trivially on the latter space, so that by condition  (I) we obtain again the condition $\sum_r \beta^{(i_N)} _r B_r^{(i_N)}  \xi = 0$. Collecting everything together we finally obtain 
\begin{align}
\label{eq:IV}
\mathrm{(I), \,(II), \,(III)} \quad & \Longrightarrow \quad \sum_r \beta^{(i_N)} _r B_r^{(i_N)}  \xi = 0.
\end{align}
Now, by \eqref{eq:I} -- \eqref{eq:III} we have
\begin{align*}
\mathrm{(I), \,(II), \,(III)} \quad & \Longleftrightarrow \quad  \gd_{\xi, \alpha^{(i_1)}, \dots , \alpha^{(i_N)} , \beta^{(i_N)}} \, ^{(i_1\dots i_N)} \phw=0.
\end{align*}
It therefore remains to study the derivatives with respect to the variables $\sigma_{i_j}$, $p^{(i_j)}$, and $\tilde v ^{(i_N)}$. It is immediately clear that 
\bqn 
\mathrm{(I)} \quad  \Longrightarrow \quad \gd_{\sigma} \, ^{(i_1\dots i_N)} \phw=0.
\eqn
By \eqref{eq:IV} we further have 
\bqn 
\mathrm{(I), \,(II), \,(III)} \quad  \Longrightarrow \quad \gd_{\tilde v ^{(i_N)}} \, ^{(i_1\dots i_N)} \phw=0.
\eqn 
Let us now take for each $j=1,\dots, N$ local coordinates $p^{(i_j)}=p^{(i_j)}(s^{(i_j)}) \in  (S_{i_1\dots i_{j-1}}^+)_{p^{(i_{j-1})}, i_j}(H_{i_j})$ around $p^{(i_j)}_0=p^{(i_j)}(0)$, and write 
\bqn 
B^{(i_N)}= g(p^{(i_N)}) \,  B_0^{(i_N)} \, g^{-1}(p^{(i_N)}) \in \g _{p^{(i_N)}}, \qquad  B_0^{(i_N)} \in \g_{p^{(i_N)}_0}, \quad g(p^{(i_N)} ) \in G.
\eqn 
On then computes 
\bqn
\frac {\gd B^{(i_N)}} {\gd{s^{(i_N)}_r}}= \Big (\frac {\gd g ( p^{(i_N)}(s^{(i_N)})) } {\gd {s^{(i_N)}_r}}\Big ) g (p^{(i_N)}(s^{(i_N)}))^{-1} B^{(i_N)} + B^{(i_N)} g(p^{(i_N)} (s^{(i_N)}))\Big (\frac {\gd g^{-1}( p^{(i_N)}(s^{(i_N)})) } {\gd {s^{(i_N)}_r}}\Big ),
\eqn
so that with  \eqref{eq:IV} one finally concludes
\bqn 
\mathrm{(I), \,(II), \,(III)} \quad  \Longrightarrow \quad \gd_{p^{(i_1)}, \dots,  p^{(i_N)}} \, ^{(i_1\dots i_N)} \phw=0.
\eqn
Thus we have computed the critical set of $\, ^{(i_1\dots i_N)} \phw$, and it remains to show that it is a $\Cinft$-submanifold of codimension $2\kappa$. For this end, let us define the subspaces
\bq
\label{eq:EF}
E^{(i_j)} = \g _{p^{(i_j)}}^\perp  \cdot x ^{(i_j \dots i_N)}, \qquad F^{(i_N)} = \g_{p^{(i_N)}}\cdot \tilde v^{(i_N)}. 
\eq
One has  $E^{(i_1)}  \subset \g   \cdot x ^{(i_1 \dots i_N)}=T_{x ^{(i_1 \dots i_N)}}(G \cdot x ^{(i_1 \dots i_N)}) $, as well as
\bqn 
E^{(i_j)}  \subset \g _{p^{(i_{j-1})}}  \cdot x ^{(i_j \dots i_N)}=T_{x ^{(i_j \dots i_N)}}(G_{p^{(i_{j-1})}} \cdot x ^{(i_j \dots i_N)}) \subset V^{(i_1 \dots i_{j-1})}
\eqn
for $2 \leq j \leq N$. Similarly, $F^{(i_N)} \subset V^{(i_1 \dots i_{N})}$. Now, for small $\tau_{i_j}$, we clearly have $E^{(i_j)} \cap V^{(i_1 \dots i_{j})}=\mklm{0}$, so that we obtain the direct sum of vector spaces
\bqn 
E^{(i_1)}\oplus E^{(i_2)}\oplus \dots  \oplus E^{(i_N)}\oplus F^{(i_N)}.
\eqn 
We therefore arrive at the characterization
\bq
\label{eq:C}
\Crit(\, ^{(i_1\dots i_N)} \phw)= \Big \{ \alpha^{(i_j)}=0, \quad \sum_r \beta^{(i_N)} _r B_r^{(i_N)} \tilde v ^{(i_N)}=0, \quad \xi \perp \Big(\bigoplus_{j=1}^N E^{(i_j)}\oplus F^{(i_N)}\Big ) \Big \},
\eq
Note that the condition $\sum_r \beta^{(i_N)} _r B_r^{(i_N)} \xi=0$ is already implied by the others. Now, for small, but arbitrary $\sigma_{i_j}$ one has
\bqn 
\dim E^{(i_j)} = \dim \g_{p^{(i_{j})}}^\perp \cdot p ^{(i_j )} = \dim G_{p^{(i_{j-1})}} \cdot p ^{(i_j )}.
\eqn
Since for  $\sigma_{i_j}\not=0$,  $1 \leq j \leq  N$, the $G$-orbit of $x^{(i_1\dots i_N)}$ is of principal type $G/ H_L$ in $\rn$ by assumption, one computes in this case
\begin{align*}
\kappa =& \dim G \cdot x^{(i_1\dots i_N)}= \dim \g \cdot x^{(i_1\dots i_N)}= \dim [  \g_{p^{(i_N)}}\oplus \g_{p^{(i_N)}}^\perp \oplus \cdots \oplus \g_{p^{(i_1)}}^\perp  ] \cdot x^{(i_1\dots i_N)} \\
=&\dim[ \g_{p^{(i_1)}}^\perp  \cdot x^{(i_1\dots i_N)} + \tau_{i_1}  \g_{p^{(i_2)}}^\perp \cdot x^{(i_2\dots i_N)} +  \tau_{i_1} \sin \tau_{i_2}  \g_{p^{(i_3)}}^\perp \cdot x^{(i_3\dots i_N)} + \dots  \\ &+  \tau_{i_1}\sin \tau_{i_2}\cdots  \sin \tau_{i_{N-1}}  \g_{p^{(i_N)}}^\perp \cdot x^{(i_N)} + \tau_{i_1}\sin \tau_{i_2}\cdots \sin \tau_{i_{N}} \g_{p^{(i_N)}} \tilde v^{(i_N)}] \\
 =& \dim[ E^{(i_1)} \oplus \tau_{i_1} E^{(i_2)}\oplus \tau_{i_1} \sin \tau_{i_2} E^{(i_3)} \oplus \dots   \oplus  \tau_{i_1}\sin \tau_{i_2}\cdots  \sin \tau_{i_{N-1}} E^{(i_N)} \\&  \oplus   \tau_{i_1}\sin \tau_{i_2}\cdots \sin \tau_{i_N} F^{(i_n)}] =\sum_{j=1} ^N \dim E^{(i_j)} + \dim F^{(i_N)}.
\end{align*}
But since the dimension of the spaces $E^{(i_j)}$ and $F^{(i_N)}$  does not depend on the variables $\sigma_{i_j}$, we obtain for sufficiently small, but arbitrary $\sigma_{i_j}$  the equality 
\bq
\label{eq:kappa}
\kappa=\sum_{j=1} ^N \dim E^{(i_j)} + \dim F^{(i_N)}.
\eq
Note that, in contrast, the dimension of  $\g \cdot x^{(i_1\dots i_N)}$ collapses, as soon as one of the $\tau_{i_j}$ becomes zero. Thus we arrive at a vector bundle with $(n-\kappa)$-dimensional fiber that is locally given by the trivialization 
\bqn 
(\sigma_{i_1}, \dots ,\sigma_{i_N}, p^{(i_1)}, \dots, p^{(i_N)},  \tilde v ^{(i_N)},  \big(\bigoplus_{j=1}^N E^{(i_j)}\oplus F^{(i_N)}\big )^\perp )\mapsto (\sigma_{i_1}, \dots \sigma_{i_N}, p^{(i_1)},\dots, p^{(i_N)},   \tilde v ^{(i_N)}).
\eqn
Consequently, by Equation \eqref{eq:C} we see that $\Crit(\, ^{(i_1\dots i_N)} \phw)$ is equal to the fiber product of the mentioned vector bundle with the isotropy algebra bundle that is given by the local trivialization
\bqn 
(\sigma_{i_1}, \dots, \sigma_{i_N}, p^{(i_1)}, \dots, p^{(i_N)},  \tilde v ^{(i_N)}, \g_{\tilde v ^{(i_N)}})\mapsto (\sigma_{i_1}, \dots, \sigma_{i_N}, p^{(i_1)}, \dots, p^{(i_N)},  \tilde v ^{(i_N)}).
\eqn
Lastly, since by Equation \eqref{eq:G} we have $\g_{\tilde v ^{(i_N)}}=\g_{x^{(i_1,\dots, i_N)}}$ in case that all $\sigma_{i_j}$ are different from zero, we necessarily have $\dim \g_{\tilde v ^{(i_N)}}=d-\kappa$, which concludes the proof of the theorem.
\end{proof}

\section{Phase analysis of the weak transform. The second fundamental theorem}

In this section, we shall prove the second fundamental theorem in the derivation of equivariant spectral asymptotics for orthogonal compact group actions. The notation will be the same as in the previous sections.
\begin{theorem}
Let
\bqn
\, ^{(i_1\dots i_N)} \tilde \psi^{tot}=\tau_{i_1} \dots \tau_{i_N} \, ^{(i_1\dots i_N)}\tilde \psi^ {wk, \, pre}=\tau_{i_1}(\sigma) \dots \tau_{i_N}(\sigma) \, ^{(i_1\dots i_N)}\tilde \psi^ {wk}
\eqn
denote the factorization of the phase function after $N$ iteration steps along the isotropy branch $ ((H_{i_1}), \dots, (H_{i_{N+1}})=(H_L))$. By construction, for $\tau_{i_j}\not=0$, $1\leq j\leq N$, the $G$-orbit through $x^{(i_1\dots i_N)}$ is of principal type $G/H_L$. Then for each point of the critical manifold $\Crit(\, ^{(i_1\dots i_N)}\tilde \psi^ {wk})$,  the restriction of 
\bqn 
\mathrm{Hess} \, ^{(i_1\dots i_N)}\tilde \psi^ {wk}
\eqn
to the normal space to $\Crit(\, ^{(i_1\dots i_N)}\tilde \psi^ {wk})$ at the given point defines a non-degenerate symmetric bilinear form.
\end{theorem}
Before proving the theorem, let us make  the following general observation. Let $M$ be a $n$-dimensional Riemannian manifold, and $C$ the critical set of a function $\psi \in \Cinft(M)$, which is assumed to be a smooth submanifold in a chart $\mathcal{O} \subset M$. Let further 
\bqn
\alpha:(x,y) \mapsto p, \qquad \beta:(m_1,\dots, m_n) \mapsto p, 	\qquad p \in \mathcal{O},
\eqn
be two systems of  local coordinates on $\mathcal{O}$, such that $\alpha(x,y) \in C$ if and only if $y=0$. One computes
\begin{align*}
\gd_{y_l}(\psi \circ \alpha)(x,y)=\sum_{i=1}^n \frac{\gd(\psi \circ \beta)}{\gd m_i} (\beta^{-1}\circ \alpha(x,y)) \, \gd _{y_l} (\beta^{-1} \circ \alpha)_i (x,y),
\end{align*}
as well as
\begin{align*}
\gd_{y_k} \gd_{y_l} (\psi \circ \alpha)(x,y)&=\sum_{i=1}^n \frac{\gd(\psi \circ \beta)}{\gd m_i} (\beta^{-1}\circ \alpha(x,y)) \, \gd_{y_k}  \gd _{y_l} (\beta^{-1} \circ \alpha)_i (x,y)\\
&+ \sum_{i,j=1}^n \frac{\gd^2(\psi \circ \beta)}{\gd m_i \gd{m_j}} (\beta^{-1}\circ \alpha(x,y))  \gd_{y_k}(\beta^{-1} \circ\alpha)_j(x,y) \,  \gd_{y_l}(\beta^{-1} \circ\alpha)_i(x,y).
\end{align*}
Since
\bqn
\alpha_{\ast,(x,y)}(\gd_{y_k})=\sum_{j=1}^n \gd_{y_k} (\beta^{-1} \circ \alpha)_j(x,y) \, \beta_{\ast,(\beta^{-1} \circ \alpha)(x,y)}(\gd_{m_j}),
\eqn
this implies
\bq
\label{eq:Hess}
\gd_{y_k} \gd_{y_l} (\psi \circ \alpha)(x,0)= \mathrm{Hess}\,  \psi_{|\alpha(x,0)} (\alpha_{\ast,(x,0)}(\gd_{y_k}),\alpha_{\ast,(x,0)}(\gd_{y_l})),
\eq
by definition of the Hessian. Let us now write $x=(x',x'')$, and consider the restriction of $\psi$ onto  the $\Cinft$-submanifold
\bdm
M_{c'}=\mklm { m \in \mathcal{O}: m=\alpha(c',x'',y)}.
\edm
We write $\psi_{c'}=\psi_{|M_{c'}}$, and denote the critical set of $\psi_{c'}$ by $C_{c'}$, which contains $C \cap M_{c'}$ as a subset. Introducing on $M_{c'}$ the local coordinates
\bqn
\alpha':(x'',y) \mapsto \alpha(c',x'',y),
\eqn
we obtain 
\bq
\gd_{y_k} \gd_{y_l} (\psi_{c'} \circ \alpha')(x'',0)= \mathrm{Hess}\,  \psi_{c'|\alpha(x'',0)} (\alpha'_{\ast,(x'',0)}(\gd_{y_k}),\alpha'_{\ast,(x'',0)}(\gd_{y_l})).
\eq
Let us now assume $C_{c'}=C \cap M_{c'}$, a transversal intersection.  Then $C_{c'}$ is a submanifold of $M_{c'}$, and the normal space to $C_{c'}$ as a submanifold of $M_{c'}$ at a point $\alpha'(x'',0)$ is spanned by the vector fields $\alpha'_{\ast,(x'',0)} (\gd _{y_k})$.
Since clearly
\bqn
\gd_{y_k} \gd_{y_l} (\psi_{c'} \circ \alpha')(x'',0)=\gd_{y_k} \gd_{y_l} (\psi \circ \alpha)(x,0),\qquad x=(c',x''),
\eqn
we thus have proven the following
\begin{lemma}
\label{lemma:A}
Assume that $C_{c'}=C \cap M_{c'}$. Then the restriction
\bqn
\mathrm{Hess} \, \psi({\alpha(c',x'',0)})_{|N_{\alpha(c',x'',0)}C}
\eqn
of the Hessian of $\psi$ to the normal space $N_{\alpha(c',x'',0)}C$ defines a non-degenerate quadratic form if, and only if the restriction
\bqn
\mathrm{Hess} \, \psi_{c'}({\alpha'(x'',0)})_{|N_{\alpha'(x'',0)}C_{c'}}
\eqn
of the Hessian of $\psi_{c'}$ to the normal space $N_{\alpha'(x'',0)}C_{c'}$ defines a non-degenerate quadratic form.
\end{lemma} \qed
\begin{proof}[Proof of second fundamental theorem] Let us begin by noting that with respect to the standard coordinates in $\rn$ and $\g$, the Hessian of $\psi$ is given by the matrix
\bqn 
\left ( \begin{array}{ccc}
0 & \eklm{ Xe_i,e_j} & -\eklm{ e_i,X_j \xi} \\
\eklm{ X e_j,e_i} & 0 & \eklm{X_j x, e_i} \\
-\eklm{e_j,X_i \xi} & \eklm{X_i x, e_j} & 0 
\end{array} \right ),
\eqn
where $\mklm{e_j}$ and $\mklm{X_j}$ denote the standard basis in $\rn$ and $\g$, respectively. A direct computation then shows that its restriction to the normal space ${N_{(x,\xi,X)} \mathrm{Reg} \,\Crit{(\psi)}}$ defines  a non-degenerate quadratic form for all $(x,\xi,X) \in \mathrm{Reg} \,\Crit{(\psi)}$. Now,   for $\sigma_{i_1} \cdots \sigma_{i_N}\not=0$,  the sequence of monoidal transformations $\zeta_{i_1} \circ \zeta_{i_1i_2} \circ \dots \circ \zeta_{i_1\dots i_N}$, composed with the transformation $\delta_{i_1\dots i_N}$,  constitutes a diffeomorphism, so that in the given charts of the resolution the restriction of 
\bqn 
\mathrm{Hess} ^{(i_1\dots i_N)} \tilde \psi^{tot} (\sigma_{i_j},p^{(i_j)},\tilde v^{(i_N)}, \alpha^{(i_j)}, \beta^{(i_N)},\xi)
\eqn
to the normal space of the total transform 
$$\tilde {\mathcal{C}}^{tot}=((\zeta_{i_1} \circ \zeta_{i_1i_2} \circ \dots \circ \zeta_{i_1\dots i_N} \circ \delta_{i_1\dots i_N}) \otimes \id_\xi)^{-1}(\Crit{(\psi)})$$
defines a non-degenerate quadratic form as well  at every point with $\sigma_{i_1} \cdots \sigma_{i_N}\not=0$.  Next, one computes
\begin{align*}
\left (\frac{\gd^2 \, ^{(i_1\dots i_N)}\tilde \psi^ {tot}}{\gd \gamma_k\gd \gamma_l}  \right )_{k,l} & = \tau_{i_1}(\sigma) \cdots  \tau_{i_N }(\sigma) 
\left (\frac{\gd^2 \, ^{(i_1\dots i_N)}\tilde \psi^ {wk}}{\gd \gamma_k\gd \gamma_l}  \right )_{k,l} \\ &+
\left (\begin{array}{cc}
\left (   \frac{ \gd ^2 (\tau_{i_1}(\sigma) \cdots  \tau_{i_N }(\sigma))}{\gd \sigma_{i_r}\sigma_{i_s}} \right )_{r,s} & 0 \\ 0 & 0
\end{array}\right ) \, ^{(i_1\dots i_N)}\tilde \psi^ {wk} +R
\end{align*}
where $R$ represents a matrix whose entries contain first order derivatives of $^{(i_1\dots i_N)}\tilde \psi^ {wk}$ as factors. But since 
\bqn 
\tilde {\mathcal{C}}^{tot}_{|\sigma_{i_1}\cdots \sigma_{i_N}\not=0} =\Crit(^{(i_1\dots i_N)}\tilde \psi^ {wk})_{|\sigma_{i_1}\cdots \sigma_{i_N}\not=0} 
\eqn
because $^{(i_1\dots i_N)}\tilde \psi^ {tot}$ and $^{(i_1\dots i_N)}\tilde \psi^ {wk}$ vanish on their critical sets, we conclude that  the transversal Hessian of $^{(i_1\dots i_N)}\tilde \psi^ {wk}$ does not degenerate along $\Crit(^{(i_1\dots i_N)}\tilde \psi^ {wk})_{|\sigma_{i_1}\cdots \sigma_{i_N}\not=0}$. Therefore, it  remains to study the transversal Hessian of $^{(i_1\dots i_N)}\tilde \psi^ {wk}$ in the case that any of the $\sigma_{i_j}$ vanishes. Now, the proof of the first fundamental theorem showed that
\bqn 
\gd _{\xi, \alpha^{(i_1)}, \dots, \alpha^{(i_N)}, \beta^{(i_N)}} \, ^{(i_1\dots i_N)}\tilde \psi^ {wk}=0 \quad \Longrightarrow \quad \gd _{\sigma_{i_1}, \dots \sigma_{i_N}, p^{(i_1)}, \dots, p^{(i_N)},\tilde v^{(i_N)}} \, ^{(i_1\dots i_N)}\tilde \psi^ {wk}=0. 
\eqn
If therefore 
$$
\, ^{(i_1\dots i_N)}\tilde \psi^ {wk}_{\sigma_{i_j}, p^{(i_j)},\tilde v^{(i_N)}}(\alpha^{(i_j)}, \beta^{(i_N)},\xi)
$$ 
denotes the weak transform of the phase function $\psi$ regarded as a function of the variables $(\alpha^{(i_1)},\dots, \alpha^{(i_N)}, \beta^{(i_N)},\xi)$ alone, while the variables $(\sigma_{i_1},\dots,\sigma_{i_N}, p^{(i_1)},\dots, p^{(i_N)},\tilde v^{(i_N)})$ are kept fixed,
\bqn 
\Crit \big ( \, ^{(i_1\dots i_N)}\tilde \psi^ {wk}_{\sigma_{i_j}, p^{(i_j)},\tilde v^{(i_N)}}\big )=\Crit \big ( \, ^{(i_1\dots i_N)}\tilde \psi^ {wk}\big )  \cap \mklm{\sigma_{i_j}, p^{(i_j)},\tilde v^{(i_N)} = \, \, \text{constant}}. 
\eqn
Thus, the critical set of $\, ^{(i_1\dots i_N)}\tilde \psi^ {wk}_{\sigma_{i_j}, p^{(i_j)},\tilde v^{(i_N)}}$ is equal to the fiber over $(\sigma_{i_j}, p^{(i_j)},\tilde v^{(i_N)})$ of the vector bundle
\bqn 
\big ((\sigma_{i_j}, p^{(i_j)},\tilde v^{(i_N)}), \g _{\tilde v^{(i_N)}} \times  \big ( \bigoplus\limits_{j=1}^N E^{(i_j)} \oplus F^{(i_N)} \big )^\perp \big ) \mapsto  (\sigma_{i_j}, p^{(i_j)},\tilde v^{(i_N)}),
\eqn
and in particular  a smooth submanifold. Lemma \ref{lemma:A} then implies that the study of the transversal Hessian  of $\, ^{(i_1\dots i_N)}\tilde \psi^ {wk}$ can be reduced to the study of the transversal Hessian of $\, ^{(i_1\dots i_N)}\tilde \psi^ {wk}_{\sigma_{i_j}, p^{(i_j)},\tilde v^{(i_N)}}$. The crucial fact is now contained in the following 
\begin{proposition}
\label{prop:1}
Assume that 
$\sigma_{i_1} \cdots \sigma_{i_N}=0$. Then 
\bqn 
\ker  \mathrm{Hess} \, ^{(i_1\dots i_N)}\tilde \psi^ {wk}_{\sigma_{i_j}, p^{(i_j)},\tilde v^{(i_N)}}(0,\dots, 0, \beta^{(i_N)},\xi)\simeq T_{(0, \dots, 0,\beta^{(i_N)},\xi)}  \mathrm{Crit} \big (\,^{(i_1\dots i_N)}\tilde \psi^ {wk}_{\sigma_{i_j}, p^{(i_j)},\tilde v^{(i_N)}} \big )
\eqn
for all $(0,\dots, 0, \beta^{(i_N)},\xi) \in \mathrm{Crit} \big (\,^{(i_1\dots i_N)}\tilde \psi^ {wk}_{\sigma_{i_j}, p^{(i_j)},\tilde v^{(i_N)}} \big )$, and arbitrary $p^{(i_j)}$, $\tilde v^{(i_j)}$.
\end{proposition}
\begin{proof}
Let us begin by computing
\begin{align*}
\gd _{\xi_r} \,^{(i_1\dots i_N)}\tilde \psi^ {wk}= &\big [\sum_{j=1}^N  \sum _{k=1}^{d^{(i_j)}} \alpha^{(i_j)}_k A^{(i_j)}_k \, p^{(i_j)}+\sum_{k=1}^{e^{(i_N)}} \beta_k^{(i_N)}  B_k^{(i_N)}  \tilde v ^{(i_N)}\big ]_r + \gd _{\xi_r} \sum _{j=1}^N O( |\tau_{i_j} A^{(i_j)}|) \\ + &  \gd _{\xi_r}  O( |\tau_{i_N} B^{(i_N)}\tilde v ^{(i_N)}|) .
\end{align*}
If $\sigma_{i_1}\cdots \sigma_{i_j}=0$, which means that all the $\tau_{i_j}$ are zero, the second derivatives read
\begin{align*}
\gd_{\xi_r} \gd _{\xi_s} \, ^{(i_1\dots i_N)}\tilde \psi^ {wk}_{\sigma_{i_j}, p^{(i_j)},\tilde v^{(i_N)}} &=0, \\
\gd_{\alpha^{(i_j)}_s} \gd _{\xi_r} \, ^{(i_1\dots i_N)}\tilde \psi^ {wk}_{\sigma_{i_j}, p^{(i_j)},\tilde v^{(i_N)}} &=[A_s^{(i_j)} p^{(i_j)}]_r, \\
\gd_{\beta^{(i_N)}_s} \gd _{\xi_r} \, ^{(i_1\dots i_N)}\tilde \psi^ {wk}_{\sigma_{i_j}, p^{(i_j)},\tilde v^{(i_N)}} &=[B_s^{(i_N)} \tilde v^{(i_N)}]_r. 
\end{align*}
Next, one has
\begin{align*}
\gd_{\alpha^{(i_j)}_s} \, ^{(i_1\dots i_N)}\tilde \psi^ {wk}&=\eklm{A_s^{(i_j)} p^{(i_j)}, \xi } +  \gd_{\alpha^{(i_j)}_s} \, O(|\tau_{i_j} A^{(i_j)}|), 
\end{align*}
so that for $\sigma_{i_1}\cdots \sigma_{i_j}=0$ all the second order derivatives involving $\alpha^{(i_j)}$  must vanish, except the ones that were already computed. Finally, the computation of the $\beta^{(i_N)}$-derivatives yields
\bqn 
\gd_{\beta^{(i_N)}_r} \gd _{\beta^{(i_N)}_s} \, ^{(i_1\dots i_N)}\tilde \psi^ {wk}_{\sigma_{i_j}, p^{(i_j)},\tilde v^{(i_N)}} =0.
\eqn
Collecting everything we see that the Hessian of the function $\, ^{(i_1\dots i_N)}\tilde \psi^ {wk}_{\sigma_{i_j}, p^{(i_j)},\tilde v^{(i_N)}}$  with respect to the coordinates $\alpha^{(i_j)}, \beta^{(i_j)},\xi$ is given on its critical set by the matrix
\bqn 
\left ( \begin{array}{ccccc}
0 & [A_s^{(i_1)} p^{(i_1)}]_r & \dots & [A_s^{(i_N)} p^{(i_N)}]_r & [B_s^{(i_N)} \tilde v^{(i_N)}]_r \\
\,  [A_r^{(i_1)} p^{(i_1)}]_s & 0 & \dots & 0 & 0 \\
\vdots & \vdots &\vdots &\vdots &\vdots \\
\, [A_r^{(i_N)} p^{(i_N)}]_s & 0 & \dots & 0 & 0 \\
 \, [B_r^{(i_N)} \tilde v^{(i_N)}]_s & 0 & \dots & 0 & 0
\end{array} \right ), 
\eqn
where   $\sigma_{i_1}\cdots \sigma_{i_j}=0$. Let us now compute the kernel of the linear transformation corresponding to this matrix. Cleary,  the vector $(\tilde \xi, \tilde \alpha^{(i_1)}, \dots, \tilde \alpha^{(i_N)}, \tilde \beta^{(i_N)})$ lies in the kernel if and only if

\medskip
\begin{tabular}{ll}
{(a)} & $\sum \tilde \alpha_r^{(i_1)} A_r^{(i_1)} p^{(i_1)}+  \dots +\sum \tilde \alpha_r^{(i_N)} A_r^{(i_N)} p^{(i_N)}+ \sum \tilde \beta_r^{(i_N)} B_r^{(i_N)} \tilde v^{(i_N)}=0$ ; \\[2pt]
{(b)} & $\eklm{Y^{(i_j)}p^{(i_j)},\tilde \xi}=0$ for all $Y^{(i_j)} \in \g_{p^{(i_j)}}^\perp, \, 1 \leq j \leq N$;\\[2pt]
{(c)} &$\eklm{Z \tilde v ^{(i_N)},\tilde \xi}=0$ for all $Z \in \g_{p^{(i_N)}}.$
\end{tabular}
\medskip

Let $V^{(i_1\dots i_N)}$, $E^{(i_j)}$ and $F^{(i_N)}$ be the subspaces in $\rn$ defined in \eqref{eq:V} and \eqref{eq:EF}. Since $\g_{p^{(i_j)}}^\perp\subset \g_{p^{(i_{j-1})}}$, the condition (b) is equivalent to $\tilde \xi \in V^{(i_1\dots i_N)}$. Next, we have
\bqn 
\sum \tilde \alpha_r^{(i_j)} A_r^{(i_1)} p^{(i_1)}+  \dots \sum \tilde \alpha_r^{(i_N)} A_r^{(i_N)} p^{(i_N)}+ \sum \tilde \beta_r^{(i_N)} B_r^{(i_N)} \tilde v^{(i_N)} \in \bigoplus_{j=1}^N E^{(i_j)} \oplus F^{(i_N)},
\eqn
so that for condition (a) to hold, it is  necessary  and sufficient that 
\bqn 
\tilde \alpha^{(i_j)} =0, \quad 1 \leq j \leq N, \qquad \sum \tilde \beta_r^{(i_N)} B_r^{(i_N)} \tilde v^{(i_N)}=0.
\eqn
In addition, condition (c) is equivalent to 
\bqn 
\tilde \xi \in N_{\tilde v ^{(i_N)}} \big ( G_{p^{(i_N)}}\cdot \tilde v ^{(i_N)} \big ).
\eqn
On the other hand, 
\begin{gather*}
T_{(0, \dots, 0,\beta^{(i_N)},\xi)}  \mathrm{Crit} \big (\,^{(i_1\dots i_N)}\tilde \psi^ {wk}_{\sigma_{i_j}, p^{(i_j)},\tilde v^{(i_N)}} \big )\\
= \Big \{( \tilde \alpha^{(i_1)}, \dots, \tilde \alpha^{(i_N)}, \tilde \beta^{(i_N)}, \tilde \xi): \tilde \alpha^{(i_j)}=0, \,  \sum \tilde  \beta^{(i_N)}_r B^{(i_N)}_r \in \g_{\tilde v^{(i_N)}}, \,  \tilde \xi \perp  \bigoplus_{j=1}^N E^{(i_j)} \oplus F^{(i_N)}\Big \}.
\end{gather*}
But since for $\sigma_{i_1}\cdots \sigma_{i_N}=0$
\bqn 
\Big ( \bigoplus_{j=1}^N E^{(i_j)} \oplus F^{(i_N)}\Big ) ^\perp = N_{\tilde v ^{(i_N)}} \big ( G_{p^{(i_N)}}\cdot \tilde v ^{(i_N)} \big ) \cap V^{(i_1\dots i_N)}= N_{\tilde v ^{(i_N)}} \big ( G_{p^{(i_N)}}\cdot \tilde v ^{(i_N)} \big ),
\eqn 
the proposition follows.
\end{proof}
Let now $B$ be a symmetric bilinear form on a finite dimensional $\mathbb{K}$-vector space $V$, and $M=(M_{ij})_{i,j}$ the corresponding Gramsian matrix with respect to a basis $\mklm{v_1,\dots,v_n}$ of $V$ such that 
\bqn 
B(u,w) = \sum_{i,j} u_i w_j M_{ij}, \qquad  u=\sum u_i v_i, \quad w=\sum w_i v_i.
\eqn
We denote the linear operator given by $M$ with the same letter, and write 
\bqn 
V =\ker M \oplus W.
\eqn
Consider the restriction $B_{|W \times W}$ of $B$ to $W\times W$, and assume that $B_{|W\times W}(u,w) =0$ for all $u \in W$, but $w\not=0$. Since the Euclidean scalar product in $V$ is non-degenerate, we necessarily must have $Mw=0$, and consequently $ w \in \ker  M \cap W=\mklm{0}$, which is a contradiction. Therefore $B_{|W \times W}$ defines a non-degenerate symmetric bilinear form. The previous proposition therefore implies that for $\sigma_{i_1}\cdots \sigma_{i_N}=0$
\bqn 
\mathrm{Hess} \,^{(i_1\dots i_N)}\tilde \psi^ {wk}_{\sigma_{i_j}, p^{(i_j)},\tilde v^{(i_N)}}(0, \dots, 0,\beta^{(i_N)},\xi)_{|N_{(0, \dots, 0,\beta^{(i_N)},\xi)}  \mathrm{Crit} \big (\,^{(i_1\dots i_N)}\tilde \psi^ {wk}_{\sigma_{i_j}, p^{(i_j)},\tilde v^{(i_N)}} \big )}
\eqn
defines a non-degenerate symmetric bilinear form for all points $(0, \dots, 0,\beta^{(i_N)},\xi)$ lying in the critical set of $\,^{(i_1\dots i_N)}\tilde \psi^ {wk}_{\sigma_{i_j}, p^{(i_j)},\tilde v^{(i_N)}}$. The second fundamental theorem now follows with Lemma \ref{lemma:A}.
\end{proof}
We are now in position to give an asymptotic description of the integral $I(\mu)$. But before, we shall say a few words about the desingularization process.

\section{Resolution of singularities and the stationary phase theorem}
\label{sec:6}

Let $M$ be a smooth variety, $\mathcal{O}_M$ the structure sheaf of rings of $M$, and $I \subset \mathcal{O}_M$ an ideal sheaf. The aim in the theory of resolution of singularities is  to construct a birational morphism $\pi: \tilde M \rightarrow M$ such that $\tilde M$ is smooth, and the pulled back ideal sheaf $\pi^\ast I$ is locally principal. This is called the \emph{principalization} of $I$, and implies resolution of singularities. That is, for every quasi-projective variety $X$, there is a smooth variety $\tilde X$, and a birational and projective morphism $\pi:\tilde X \rightarrow X$. Vice versa, resolution of singularities implies principalization.

Consider next the derivative  $D(I)$ of $I$, which is the sheaf ideal that is generated by all derivatives of elements of $I$. Let further $Z \subset M$ be a smooth subvariety, and $\pi: B_Z M \rightarrow M$ the corresponding monoidal transformation with exceptional divisor $F \subset B_ZM$. Assume that $(I,m)$ is a  marked ideal sheaf with $m \leq \mathrm{ord}_Z I$. The total transform $\pi^\ast I$ vanishes along $F$ with multiplicity $\mathrm{ord}_Z I$, and by removing the ideal sheaf $\mathcal{O}_{B_ZM}(-\mathrm{ord}_Z I \cdot F)$ from $\pi^\ast I$  we obtain the \emph{birational, or weak transform} $\pi_\ast^{-1} I$ of $I$. Take local coordinates $(x_1,\dots, x_n)$ on $M$ such that $Z=(x_1= \dots =x_r=0)$. As a consequence,
\bqn
y_1=\frac {x_1}{x_r}, \dots, y_{r-1}=\frac{x_{r-1}}{x_r}, y_r=x_r, \dots,  y_n=x_n
\eqn
define local coordinates on $B_Z M$, and for $(f,m) \in (I,m)$ one has
\bqn
\pi_\ast^{-1} (f(x_1,\dots,x_n),m)= (y_r^{-m} f (y_1y_r, \dots y_{r-1}y_r, y_r, \dots, y_n),m).
\eqn
By computing the first derivatives of $\pi_\ast^{-1} (f(x_1,\dots,x_n),m)$, one then sees that for any composition  $\Pi:\tilde M \rightarrow M$ of blowing-ups of order greater or equal than $m$, 
\bqn
\Pi^{-1}_\ast ( D(I,m)) \subset D(\Pi^{-1}_\ast(I,m)),
\eqn
see Koll\'{a}r \cite{kollar}, Theorem 71. 

Let us now come back to our situation, and consider on $T^\ast \rn\times \g$ the ideal $I_\psi=(\psi)$ generated by the phase function $\psi=\mathbb{J}(x,\xi)(X)=\eklm{Xx,\xi}$, together with its vanishing set $V_\psi$. The derivative of $I$ is given by $D(I_\psi) = I_{\mathcal{C}}$, 
where $I_{\mathcal{C}}$ denotes the vanishing ideal of the critical set $\mathcal{C}=\Crit(\psi)$, and by the implicit function theorem $\mathrm{Sing} \,V_\psi \subset V_\psi \cap \mathcal{C}=\mathcal{C}$. Let  $((H_{i_1}), \cdots, (H_{i_{N+1}})=(H_L))$ be an arbitrary branch of isotropy types, and consider the corresponding sequence of monoidal transformations $(\zeta_{i_1} \circ \zeta_{i_1i_2} \circ \cdots \circ \zeta_{i_1\dots i_{N}}) \otimes \id_\xi$. Compose it with the sequence of monoidal transformations $\delta_{i_1\dots i_N}$, and denote the resulting transformation by $\zeta$. We then have the diagram

\bqn 
 \begin{array}{ccccc}
\zeta^\ast (I_{\mathcal{C}})  & \supset & \zeta^\ast (I_\psi)  &= \prod_{i=1}^N \tau_{i_j}(\sigma) \, \cdot\zeta^{-1}_\ast(I_\psi) & \ni \tau_{i_1}(\sigma) \cdots \tau_{i_N}(\sigma)  \,^{(i_1\dots i_N)}\tilde \psi^ {wk}  \\ [4pt]
\uparrow &  & \uparrow  & &\\[4pt]
I_{\mathcal{C}}  &\supset & I_\psi   & \ni \psi & 
\end{array} 
\eqn

\medskip
\noindent 
According to the previous considerations, we have the inclusion 
\bqn 
\zeta^{-1}_\ast (I_{\mathcal{C}}) \subset D(\zeta^{-1}_\ast (I_\psi)).
\eqn
Furthermore, the first fundamental theorem implies that $D(\zeta^{-1}_\ast (I_\psi))$ is a resolved ideal. Nevertheless, it is easy to see that $\zeta^{-1}_\ast(I_\psi)$ is not resolved, so that $\prod_{i=1}^N \tau_{i_j}(\sigma) \, \cdot\zeta^{-1}_\ast(I_\psi)$ is only a partial principalization. Let us now consider the set 
\begin{align*} 
\tilde{\mathcal{C}}&^{(i_1\dots i_N)}=\mathrm{Cl}\mklm{(\sigma_{i_j}, p^{(i_j)},  \tilde v ^{(i_N)}, \alpha^{(i_j)},  \beta ^{(i_N)}, \xi):   \sigma_{i_1}\cdots \sigma_{i_N} \not=0,  \quad (x^{(i_1\dots i_N)},\xi,X^{(i_1\dots i_N)} )\in   \mathcal{C}} \\
&=\mathrm{Cl}\mklm{(\sigma_{i_j}, p^{(i_j)}, \tilde v ^{(i_N)}, \, 0, \beta^{(i_N)}, \xi):  \sigma_{i_1}\cdots \sigma_{i_N} \not=0,  \quad  B^{(i_N)} \in \g_{(\tilde v^{(i_N)},\xi)}, \quad \xi \perp (\g \cdot x^{(i_1\dots i_N)}) } \\
&=\mathrm{Cl}\mklm{(\sigma_{i_j}, p^{(i_j)},  \tilde v ^{(i_N)}, \, 0, \beta^{(i_N)}, \xi):  \sigma_{i_1}\cdots \sigma_{i_N} \not=0,  \quad  B^{(i_N)} \in \g_{\tilde v^{(i_N)}}, \quad \xi \perp \bigoplus_{l=1}^N E^{(i_l)} \oplus F^{(i_N)} }, 
\end{align*}
where we made use of the decomposition
\bqn 
\g \cdot x^{(i_1\dots i_N)}= E^{(i_1)} \oplus \tau_{i_1} E^{(i_2 )}\oplus  \tau_{i_1} \sin \tau_{i_2}  E^{(i_3)}\oplus \dots \oplus   \tau_{i_1} \sin \tau_{i_2} \cdots \sin \tau_{i_N}  F^{(i_N)}
\eqn
and took into account that $G_{\tilde v^{(i_N)}}$ acts trivially on $\big ( \bigoplus_{l=1}^N E^{(i_l)} \oplus F^{(i_N)}\big )^\perp $. Equation \eqref{eq:C} then implies that 
\bqn 
\tilde{\mathcal{C}}^{(i_1\dots i_N)} = \Crit ( \,^{(i_1\dots i_N)}\tilde \psi^ {wk}).
\eqn
Nevertheless, this does not result in a resolution $\tilde{\mathcal{C}}$  of $\mathcal{C}$, but only in a partial resolution, since the induced global birational transform $ \tilde{\mathcal{C}} \rightarrow \mathcal{C}$ is not surjective in general. This is because the centers of our monoidal transformations were only chosen in $\rn_x \times \g$, to keep the phase analysis of the weak transform of $\psi$ as simple as possible. In turn, the $\xi$-singularities of $\mathcal{C}$ were not completely resolved.

As we shall see in the next section, the principalization of the ideal $I_\psi$ 
\bqn 
\zeta^\ast (I_\psi)  = \tau_{i_1}\cdots \tau_{i_N} \zeta^{-1}_\ast(I_\psi), 
\eqn
and the fact that the weak transform $ \,^{(i_1\dots i_N)}\tilde \psi^ {wk}$ has a clean critical set, are essential for an application of the stationary phase principle in the context of singular equivariant asymptotics. By Hironaka's theorem on resolution of singularities,  a resolution $\zeta$ of the vanishing set of  $\psi$  always exists, which is equivalent to the principalization of the ideal $I_\psi$.  But in general, such a resolution would not be explicit enough \footnote{In particular, the so-called numerical data of $\zeta$ are not known a priori, which in our case are given in terms of the dimensions $c^{(i_j)}$ and $d^{(i_j)}$.}  to allow an application of the stationary phase theorem. This is the reason why we were forced to construct an explicit, though partial, resolution $\zeta$ of $\mathcal{C}$ in $T^\ast \rn\times \g $, using as centers  isotropy algebra bundles over sets of maximal singular orbits. Partial desingularizations  of the zero level set $\Omega$ of the moment map and the symplectic quotient $\Omega/G$ have been obtained  e.g. by Meinrenken-Sjamaar \cite{meinrenken-sjamaar} for compact symplectic manifolds with a Hamiltonian compact Lie group action by performing blowing-ups along minimal symplectic suborbifolds containing the strata of maximal depth  in $\Omega$.

\section{Asymptotics for the integrals  $I_{i_1\dots i_N}(\mu)$ }

In this section, we will give an asymptotic description of the integrals $I_{i_1\dots i_N}(\mu)$  defined in \eqref{eq:N}.  Since the considered integrals are absolutely convergent integral, we can interchange the order of integration by Fubini, and write
\bqn
I_{i_1\dots i_N}(\mu)= \int_{(-T,T)^N}  J_{\tau_{i_1}, \dots, \tau_{i_N}}\Big ( \frac \mu{\tau_{i_1}\cdots \tau_{i_N}} \Big ) \prod_{j=1}^N |\tau_{i_j}|^{c^{(i_j)} + \sum _{r=1}^j d^{(i_r)} -1} \d \tau_{i_N} \dots \d \tau_{i_1},
\eqn
where we set
\bqn 
 J_{\tau_{i_1}, \dots, \tau_{i_N}}(\nu)= \int  e^{i \, ^{(i_1\dots i_N)} \tilde \psi ^{wk,pre}/\nu}  \, a_{i_1\dots i_N} \,    \Phi_{i_1\dots i_N}  \d \xi  dA^{(i_1)} \dots  \d A^{(i_N)} \d B^{(i_N)} \d \tilde v^{(i_N)}  \d p^{(i_{N})} \dots   \d p^{(i_{1})},
\eqn
and introduced the new parameter
\bqn
\nu =\frac \mu {\tau_{i_1}\cdots \tau_{i_N}}.
\eqn
Now,  for an arbitrary $0<\epsilon < T$ to be chosen later we define
\begin{align*}
I_{i_1\dots i_N}^1(\mu)&= \int_{((-T,T)\setminus (-\epsilon,\epsilon))^N}  J_{\tau_{i_1}, \dots, \tau_{i_N}}\Big ( \frac \mu{\tau_{i_1}\cdots \tau_{i_N}} \Big ) \prod_{j=1}^N |\tau_{i_j}|^{c^{(i_j)} + \sum _{r=1}^j d^{(i_r)} -1} \d \tau_{i_N} \dots \d \tau_{i_1},\\
I_{i_1\dots i_N}^2(\mu)&= \int_{(-\epsilon,\epsilon)^N}  J_{\tau_{i_1}, \dots, \tau_{i_N}}\Big ( \frac \mu{\tau_{i_1}\cdots \tau_{i_N}} \Big ) \prod_{j=1}^N |\tau_{i_j}|^{c^{(i_j)} + \sum _{r=1}^j d^{(i_r)} -1} \d \tau_{i_N} \dots \d \tau_{i_1} . 
\end{align*}
\begin{lemma}
\label{lemma:kappa}
One has $c^{(i_j)} + \sum _{r=1}^j d^{(i_r)} -1\geq \kappa$ for arbitrary $j=1,\dots, N$.
\end{lemma}
\begin{proof}
We first note that
\bqn
c^{(i_j)} = \dim (\nu_{i_1\dots i_j})_{p^{(i_j)}} \geq \dim G_{p^{(i_j)}} \cdot x^{(i_{j+1}\dots x_N)} +1.
\eqn
Indeed, $(\nu_{i_1\dots i_j})_{p^{(i_j)}}$ is an orthogonal $G_{p^{(i_j)}}$-space, so that the dimension of the $G_{p^{(i_j)}}$-orbit of $x^{(i_{j+1}\dots x_N)}\in (S^+_{i_1 \dots i_j})_{p^{(i_j)}}$ can be at most $c^{(i_j)}-1$. Now, under the assumption $\sigma_{i_1}\cdots \sigma_{i_N}\not=0$ one computes
\begin{align*}
\dim G&_{p^{(i_j)}}  \cdot x^{(i_{j+1}\dots i_N)}= \dim \g_{p^{(i_j)}} \cdot x^{(i_{j+1}\dots i_N)}= \dim [  \g_{p^{(i_N)}}\oplus \g_{p^{(i_N)}}^\perp \oplus \cdots \oplus \g_{p^{(i_{j+1})}}^\perp  ] \cdot x^{(i_{j+1}\dots i_N)} \\
=&\dim[ \g_{p^{(i_{j+1})}}^\perp  \cdot x^{(i_{j+1}\dots i_N)} + \sin \tau_{i_{j+1}}  \g_{p^{(i_{j+2})}}^\perp \cdot x^{(i_{j+2}\dots i_N)} + \cdots  \\ 
&+  \sin \tau_{i_{j+1}} \cdots \sin \tau_{i_{N-1}}  \g_{p^{(i_N)}}^\perp \cdot x^{(i_N)} +  \sin \tau_{i_{j+1}} \cdots \sin \tau_{i_{N}} \g_{p^{(i_N)}} \tilde v^{(i_N)}] \\ 
=& \dim[ E^{(i_{j+1})} \oplus \sin \tau_{i_{j+1}} E^{(i_{j+2})}\oplus \cdots  \oplus \sin \tau_{i_{j+1}} \cdots \sin \tau_{i_{N-1}} E^{(i_N)} \oplus  \sin \tau_{i_{j+1}} \cdots \sin \tau_{i_N} F^{(i_n)}] \\   
=&\sum_{l=j+1} ^N \dim E^{(i_l)} + \dim F^{(i_N)},
\end{align*}
which implies 
\bqn
c^{(i_j)} \geq \sum_{l=j+1} ^N \dim E^{(i_l)} + \dim F^{(i_N)} +1.
\eqn
Here we used the same arguments as in the proof of Equation \eqref{eq:kappa}. 
On the other hand, one has
\begin{align*}
d^{(i_j)}&=\dim \g_{p^{(i_j)}}^\perp =\dim \g_{p^{(i_j)}}^\perp \cdot p^{(i_j)} =\dim \g_{p^{(i_j)}}^\perp  \cdot x^{(i_j\dots i_N)} = \dim E^{(i_j)},
\end{align*}
since $x^{(i_j\dots i_N)}$ lies in a slice around $G_{p^{(i_{j-1})}}\cdot p^{(i_j)}$. The assertion now follows with \eqref{eq:kappa}.
\end{proof}
As a consequence of the lemma, we obtain for $I_{i_1\dots i_N}^2(\mu)$ the estimate
\begin{align}
\label{eq:I2}
\begin{split}
I_{i_1\dots i_N}^2(\mu)&\leq C \int_{(-\epsilon,\epsilon)^N}  \prod_{j=1}^N |\tau_{i_j}|^{c^{(i_j)} + \sum _{r=1}^j d^{(i_r)} -1} \d \tau_{i_N} \dots \d \tau_{i_1} \\
&\leq C \int_{(-\epsilon,\epsilon)^N}  \prod_{j=1}^N |\tau_{i_j}|^{\kappa} \d \tau_{i_N} \dots \d \tau_{i_1} =\frac{2C}{\kappa+1} \epsilon^{N (\kappa+1)}
\end{split}
\end{align}
for some $C>0$. Let us now turn to the integral $I_{i_1\dots i_N}^1(\mu)$. After performing the change of variables $\delta_{i_1\dots i_N}$
one obtains
\begin{align*}
I_{i_1\dots i_N}^1(\mu)&  = \int\limits _{\epsilon < |\tau_{i_j} (\sigma)|< T} \hspace{-.5cm} J_{\sigma_{i_1}, \dots, \sigma_{i_N}}\Big ( \frac \mu{\tau_{i_1}(\sigma)\cdots \tau_{i_N}(\sigma)} \Big )  \prod_{j=1}^N |\tau_{i_j}(\sigma)|^{c^{(i_j)} + \sum _{r=1}^j d^{(i_r)} -1} \, |\det D\delta_{i_1\dots i_N}(\sigma) | \d \sigma,
\end{align*}
where 
\bqn
 J_{\sigma_{i_1}, \dots, \sigma_{i_N}}(\nu)= \int  e^{i \, ^{(i_1\dots i_N)} \tilde \psi ^{wk}_\sigma /\nu}  \, a_{i_1\dots i_N} \,    \Phi_{i_1\dots i_N}  \d \xi   dA^{(i_1)} \dots  \d A^{(i_N)}  \d B^{(i_N)} \d \tilde v^{(i_N)} \d p^{(i_{N})} \dots   \d p^{(i_{1})}.
\eqn
Here we denoted by $ ^{(i_1\dots i_N)} \tilde \psi ^{wk}_\sigma $ the weak transform of the phase function $\psi$ as a function of the variables $ p^{(i_j)},  \tilde v^{(i_N)}, \alpha^{(i_j)}, \beta^{(i_N)}$ alone, while the variables $\sigma=(\sigma_{i_1},\dots \sigma_{i_N})$ are regarded as parameters. The idea is now to make use of the principle of the stationary phase to give an asymptotic expansion of $J_{\sigma_{i_1}, \dots, \sigma_{i_N}}(\nu)$. 

\begin{theorem}[Generalized stationary phase theorem for manifolds]
\label{thm:SP}
Let $M$ be a  $n$-dimensional Riemannian manifold,  $\psi \in \Cinft(M)$ be a real valued phase function,  $\mu >0$, and set
\bqn
I({\mu})=\int_M e^{i\psi(m)/\mu} a(m) \, dm,
\eqn
where $a(m)dm$ denotes  a compactly supported $\Cinft$-density on $M$. Let
$$\mathcal C=\mklm{m \in M: \psi_\ast:TM_m \rightarrow T\R_{\psi(m)} \text{ is zero}}$$
 be the  critical set of the phase function $\psi$, and assume that
\begin{enumerate}
\item $\mathcal{C}$ is a smooth submanifold of $M$  of dimension $p$ in a neighborhood of the support of $a$;
\item for all $m \in \mathcal{C}$, the restriction $\psi''(m)_{|N_m\mathcal{C}}$ of the Hessian of $\psi$ at the point $m$  to the normal space $N_m\mathcal{C}$ is a non-degenerate quadratic form.
\end{enumerate}
\noindent
Then, for all $N \in \N$, there exists a constant $C_{N,\psi}>0$ such that
\bqn
|I(\mu) - e^{i\psi_0/\mu}(2\pi \mu)^{\frac {n-p}{2}}\sum_{j=0} ^{N-1} \mu^j Q_j (\psi;a)| \leq C_{N,\psi} \mu^N \vol (\supp a \cap \mathcal{C}) \sup _{l\leq 2N} \norm{D^l a }_{\infty,M},
\eqn
where $D^l$ is a differential operator on $M$ of order $l$, and $\psi_0$ is the constant value of $\psi$ on $\mathcal{C}$. Furthermore, for each $j$ there exists a constant $\tilde C_{j,\psi}>0$ such that 
\bqn
|Q_j(\psi;a)|\leq \tilde C_{j,\psi}  \vol (\supp a \cap \mathcal{C}) \sup _{l\leq 2j} \norm{D^l a }_{\infty,\mathcal{C}},
\eqn
and, in particular,
\bqn
Q_0(\psi;a)= \int _{\mathcal{C}} \frac {a(m)}{|\det \psi''(m)_{|N_m\mathcal{C}}|^{1/2}} d\sigma_{\mathcal{C}}(m) e^{ i \pi\sigma_{\psi''}},
\eqn
where $\sigma_{\psi''}$ is the constant value of the signature of $\psi''(m)_{|N_m\mathcal{C}}$ for $m$ in $\mathcal{C}$.
\end{theorem}
\begin{proof}
See for instance H\"{o}rmander, \cite{hoermanderI}, Theorem 7.7.5, together with Combescure-Ralston-Robert \cite{combescure-ralston-robert}, Theorem 3.3.
\end{proof}
\begin{remark}
\label{rmk:A}
An examination of the proof of the foregoing theorem shows that the constants $C_{N,\psi}$ are essentially bounded from above by 
\bqn 
\sup_{m \in \mathcal{C} \cap \supp a} \norm {\Big ( \psi''(m)_{|N_m\mathcal{C}}\Big ) ^{-1}}.
\eqn
Indeed, let $\alpha:(x,y) \rightarrow m \in \mathcal{O}\subset M$ be local normal coordinates such that $\alpha(x,y) \in \mathcal{C}$ if, and only if, $y=0$. By \eqref{eq:Hess}, the transversal Hessian $\mathrm{Hess} \, \psi(m)_{|N_m\mathcal{C}}$ is given in these coordinates by the matrix
\bqn 
\Big (\gd _{y_k} \gd _{y_l} (\psi \circ \alpha)(x,0) \Big )_{k,l} \
\eqn
where $m =\alpha (x,0)$. If now the transversal Hessian of $\psi$ is non-degenerate at the point $m=\alpha(x,0)$, then $y=0$ is a non-degenerate critical point of the function  $y \mapsto (\psi \circ \alpha)(x,y)$, and therefore an isolated critical point by the lemma of Morse. As a consequence,  
\bq
\label{eq:est}
\frac{|y|}{|\gd_y (\psi \circ \alpha)(x,y)|} \leq 2 \norm {\Big (\gd_{y_k} \gd _{y_l} (\psi\circ \alpha)(x,0)\Big ) _{k,l}^{-1}}
\eq
for $y$ close to zero. The assertion now follows by applying H\"{o}rmander \cite{hoermanderI}, Theorem 7.7.5,  to the integral
\bqn 
\int_{\alpha^{-1}(\mathcal{O})} e^{i (\psi \circ \alpha) (x,y)/\mu} (a \circ \alpha)(x,y) \d y \d x
\eqn
in the variable $y$, and with $x$ as a parameter, since in our situation the constant $C$ occuring in H\"{o}rmander \cite{hoermanderI}, Equation (7.7.12), is precisely bounded by \eqref{eq:est}, if we assume as we may that $a$ is supported near $\mathcal{C}$. A similar observation holds with respect to the constants $\tilde C_{j,\psi}$.
\end{remark}
 
Now, as a consequence of the fundamental theorems, and Lemma \ref{lemma:A}, together with the observations preceding Proposition \ref{prop:1},
\begin{itemize}
\item the critical set $\Crit( ^{(i_1\dots i_N)} \tilde \psi ^{wk}_\sigma )$ is a $\Cinft$-submanifold of codimension $2\kappa$ for arbitrary $\sigma$;
\item  the transversal Hessian
\bqn
\mathrm{Hess}  \,^{(i_1\dots i_N)} \tilde \psi ^{wk}_\sigma ( p^{(i_j)},  \tilde v^{(i_N)}, \alpha^{(i_j)}, \beta^{(i_N)})_{|N_{  (p^{(i_j)},  \tilde v^{(i_N)}, \alpha^{(i_j)}, \beta^{(i_N)})}\Crit \big ( ^{(i_1\dots i_N)} \tilde \psi ^{wk}_\sigma \big)}
\eqn 
defines a non-degenerate symmetric bilinear form for arbitrary $\sigma$ at every point  of the critical set of $^{(i_1\dots i_N)} \tilde \psi ^{wk}_\sigma$.
\end{itemize}
Thus, the necessary conditions for applying the principle of the stationary phase to the integral $J_{\sigma_{i_1},\dots,\sigma_{i_N}}(\nu)$ are fulfilled, and we arrive at  the following
\begin{theorem}
\label{thm:J}
Let $\sigma=(\sigma_{i_1},\dots, \sigma_{i_N})$ be a fixed set of parameters. Then, for every $\tilde N \in \N$ there exists a constant $C_{\tilde N, ^{(i_1\dots i_N)} \tilde \psi ^{wk}_\sigma}>0$ such that 
\bqn
|J_{\sigma_{i_1},\dots ,\sigma_{i_N}} (\nu) -(2\pi |\nu|)^\kappa\sum_{j=0} ^{\tilde N-1} |\nu|^j Q_j (^{(i_1\dots i_N)} \tilde \psi ^{wk}_\sigma;a_{i_1\dots i_N} \Phi_{i_1\dots i_N})| \leq C_{\tilde N,^{(i_1\dots i_N)} \tilde \psi ^{wk}_\sigma} |\nu|^{\tilde N},
\eqn
with estimates for the coefficients $Q_j$, and an explicit expression for $Q_0$.
\end{theorem}\qed

Before going on, let us remark that for the computation of the integrals $I_{i_1\dots i_N}^1(\mu)$ it is only necessary to have an asymptotic expansion for the integrals $J_{\sigma_{i_1},\dots,\sigma_{i_N}} (\nu)$ in the case that $\sigma_{i_1} \cdots \sigma_{i_N}\not=0$, which can also be  obtained without the fundamental theorems using only the factorization of the phase function $\psi$ given by the resolution process. Nevertheless, the main consequence to be drawn from the fundamental theorems is that the constants $C_{\tilde N, ^{(i_1\dots i_N)} \tilde \psi ^{wk}_\sigma}$ and the coefficients $Q_j$ in Theorem \ref{thm:J} have uniform bounds in $\sigma$. Indeed, by Remark \ref{rmk:A} we have
\bqn 
C_{\tilde N, ^{(i_1\dots i_N)} \tilde \psi ^{wk}_\sigma} \leq C'_{\tilde N} \sup_{ p^{(i_j)},  \tilde v^{(i_N)}, \alpha^{(i_j)}, \beta^{(i_N)}} \norm{\Big ({\mathrm{Hess} \,  ^{(i_1\dots i_N)} \tilde \psi ^{wk}_\sigma}_{|N \Crit ( ^{(i_1\dots i_N)} \tilde \psi ^{wk}_\sigma)}\Big ) ^{-1}}.
\eqn
But since by Lemma \ref{lemma:A} the transversal Hessian 
\bqn 
{\mathrm{Hess} \,  ^{(i_1\dots i_N)} \tilde \psi ^{wk}_\sigma}_{|N_{ (p^{(i_j)},  \tilde v^{(i_N)}, \alpha^{(i_j)}, \beta^{(i_N)} ) } \Crit ( ^{(i_1\dots i_N)} \tilde \psi ^{wk}_\sigma)}
\eqn 
is given by 
\bqn 
{\mathrm{Hess} \,  ^{(i_1\dots i_N)} \tilde \psi ^{wk}}_{|N_{ (\sigma_{i_j}, p^{(i_j)},  \tilde v^{(i_N)}, \alpha^{(i_j)}, \beta^{(i_N)} ) } \Crit ( ^{(i_1\dots i_N)} \tilde \psi ^{wk})},
\eqn 
we finally obtain the estimate
\bqn 
C_{\tilde N, ^{(i_1\dots i_N)} \tilde \psi ^{wk}_\sigma} \leq C'_{\tilde N} \sup_{\sigma_{i_j},  p^{(i_j)},  \tilde v^{(i_N)}, \alpha^{(i_j)}, \beta^{(i_N)}} \norm{\Big ({\mathrm{Hess} \,  ^{(i_1\dots i_N)} \tilde \psi ^{wk}}_{|N \Crit ( ^{(i_1\dots i_N)} \tilde \psi ^{wk})}\Big ) ^{-1}} \leq C_{\tilde N, {(i_1\dots i_N)}}
\eqn
by a constant independent of $\sigma$. Similarly, one can show the existence of bounds of the form
\bqn 
|Q_j(^{(i_1\dots i_N)} \tilde \psi ^{wk}_\sigma; a_{i_1\dots i_N} \Phi_{i_1\dots i_N}) |\leq \tilde C_{j, {(i_1\dots i_N)}},
\eqn
with constants $\tilde C_{j, {(i_1\dots i_N)}}$ independent of $\sigma$. As a consequence of Theorem \ref{thm:J}, we obtain for arbitrary $\tilde N\in \N$
\begin{align*}
|J_{\sigma_{i_1},\dots, \sigma_{i_N}} (\nu) & -(2\pi |\nu|)^\kappa Q_0( ^{(i_1\dots i_N)} \tilde \psi ^{wk}_\sigma;a_{i_1\dots i_N} \Phi_{i_1\dots i_N})| \\&\leq 
\Big |J_{\sigma_{i_1},\dots ,\sigma_{i_N}} (\nu) -(2\pi |\nu|)^\kappa\sum_{l=0} ^{\tilde N-1} |\nu|^l Q_l (^{(i_1\dots i_N)} \tilde \psi ^{wk}_\sigma;a_{i_1\dots i_N} \Phi_{i_1\dots i_N})\Big | \\ & + (2\pi |\nu|)^\kappa \sum_{l=1} ^{\tilde N-1} |\nu|^l |Q_l (^{(i_1\dots i_N)} \tilde \psi ^{wk}_\sigma;a_{i_1\dots i_N} \Phi_{i_1\dots i_N})| \leq c_1 |\nu|^{\tilde N}+c_2 |\nu|^\kappa \sum_{l=1}^{\tilde N-1} |\nu|^l
\end{align*}
with constants $c_i>0$ independent  of both $\sigma$ and $\nu$.
From this we deduce
\begin{align*}
\Big |I_{i_1\dots i_N}^1&(\mu)  - (2\pi \mu)^{\kappa}  \int _{\epsilon < |\tau_{i_j}(\sigma) |< T}   Q_0 \prod_{j=1}^N |\tau_{i_j}(\sigma)|^{c^{(i_j)} + \sum _{r=1}^j d^{(i_r)} -1-\kappa}  |\det D\delta_{i_1\dots i_N}(\sigma) | \d \sigma \Big| \\&\leq c_3 \mu^{\tilde N}   \int_{\epsilon < |\tau_{i_j}(\sigma) |< T}   \prod_{j=1}^N |\tau_{i_j}(\sigma)|^{c^{(i_j)} + \sum _{r=1}^j d^{(i_r)} -1-\tilde N} \, |\det D\delta_{i_1\dots i_N}(\sigma) | \d \sigma\\
&+ c_4 \mu^{\kappa} \sum_{l=1}^{\tilde N-1} \mu ^l  \int_{\epsilon < |\tau_{i_j}(\sigma) |< T}   \prod_{j=1}^N |\tau_{i_j}(\sigma)|^{c^{(i_j)} + \sum _{r=1}^j d^{(i_r)} -1-\kappa -l} \, |\det D\delta_{i_1\dots i_N}(\sigma) | \d \sigma\\
& \leq c_5 \mu^{\tilde N} \max \Big \{1, \prod _{j=1}^N  \epsilon^{c^{(i_j)} + \sum _{r=1}^j d^{(i_r)} -\tilde N} \Big \} +c_6 \sum_{l=1}^{\tilde N -1} \mu^{\kappa +l} \max \Big \{1, \prod _{j=1}^N  \epsilon^{c^{(i_j)} + \sum _{r=1}^j d^{(i_r)} -\kappa -l} \Big \}.
\end{align*}
We now set
\bqn 
\epsilon=\mu^{1/N}.
\eqn
Taking into account Lemma \ref{lemma:kappa}, one infers  that the right hand side of the last inequality can be estimated by 
\bqn 
c_5 \max \mklm{ \mu^{\tilde N}, \mu^{\kappa +1}} + c_6 \sum _{l=1}^{\tilde N-1} \max \mklm{ \mu^{\kappa +l}, \mu^{\kappa +1}},
\eqn
so that for sufficiently large $\tilde N \in \N$ we finally obtain an asymptotic expansion for $I_{i_1\dots i_N}(\mu)$ by taking into account \eqref{eq:I2}, and the fact that 
\bqn 
(2\pi \mu)^\kappa \int_{0 < |\tau_{i_j} |< \mu^{1/N}}  Q_0  \prod_{j=1}^N |\tau_{i_j}|^{c^{(i_j)} + \sum _{r=1}^j d^{(i_r)} -1} \d \tau_{i_N} \dots \d \tau_{i_1} = O(\mu^{\kappa+1}). 
\eqn
\begin{theorem}
Let the assumptions of the first fundamental theorem be fulfilled. Then 
\bqn 
I_{i_1\dots i_N}(\mu)=(2 \pi \mu)^\kappa L_{i_1\dots i_N}+ O(\mu^{\kappa+1}),
\eqn
where the leading coefficient $L_{i_1\dots i_N}$ is given by 
\bq
\label{eq:L}
L_{i_1\dots i_N}=\int_{\Crit( ^{(i_1\dots i_N)} \tilde \psi^{wk})} \frac { a_{i_1\dots i_N} \Phi_{i_1\dots i_N} \, d\Crit( ^{(i_1\dots i_N)} \tilde \psi^{wk})} {|\mathrm{Hess} ( ^{(i_1\dots i_N)} \tilde \psi^{wk})_{N\Crit( ^{(i_1\dots i_N)} \tilde \psi^{wk})}|^{1/2}}.
\eq
\end{theorem}\qed

\section{Statement of the main result}
\label{sec:8}

Let us now return to our departing point, that is, the asymptotic behavior of the integral $I(\mu)$ introduced in \eqref{int}. For this, we still have to examine the contributions to $I(\mu)$ coming from integrals of the form
\begin{align}
\label{eq:P}
\begin{split}
\tilde I_{i_1\dots i_\Theta}(\mu)&=\int_{M_{i_1}(H_{i_1})\times (-T,T)} \Big [ \int_{(S_{i_1})_{p^{(i_1)},i_2}(H_{i_2})\times (-T,T)} \dots \Big [  \int_{(S_{i_1\dots i_{\Theta-1}})_{p^{(i_{\Theta-1})},i_{\Theta}}(H_{i_{\Theta}})\times (-T,T)} \\
&\Big [ \int_{(S_{i_1\dots i_{\Theta}}^+)_{p^{(i_\Theta)}}\times \g_{p^{(i_{\Theta})}}\times \g_{p^{(i_{\Theta})}}^\perp \times \cdots \times \g_{p^{(i_{1})}}^\perp\times \rn} e^{i\frac {\tau_1 \dots \tau_\Theta}\mu \, ^{(i_1\dots i_\Theta)} \tilde \psi ^{wk}}  \, a_{i_1\dots i_\Theta} \,   \tilde \Phi_{i_1\dots i_\Theta} \\
& \d \xi  \d A^{(i_1)} \dots  \d A^{(i_\Theta)}  \d B^{(i_\Theta)} \d \tilde v^{(i_\Theta)} \Big ]  \d \tau_{i_\Theta} \d p^{(i_{\Theta})} \dots  \Big ] \d \tau_{i_2} \d p^{(i_{2})} \Big ]\d \tau_{i_1} \d p^{(i_{1})}.
\end{split}
\end{align}
where 
\bqn 
a_{i_1\dots i_\Theta}=[ a \, \chi_{i_1}  \circ (\id_\xi \otimes \zeta_{i_1} \circ \zeta_{i_1i_2} \circ \dots \circ \zeta_{i_1\dots i_\Theta})] \, [ \chi_{i_1i_2} \circ \zeta_{i_1i_2} \circ \dots\circ \zeta_{i_1\dots i_\Theta}  ] \dots [\chi_{i_1\dots i_\Theta} \circ \zeta_{i_1\dots i_\Theta}]
\eqn
is supposed to have compact support in one of the $\alpha^{(i_\Theta)}$-charts, and
\begin{align*}
\tilde \Phi_{i_1\dots i_\Theta} &=\prod_{j=1}^\Theta |\tau_{i_j}|^{c^{(i_j)}+\sum_r^j d^{(i_r)}-1}\Phi_{i_1\dots i_\Theta},
\end{align*}
 $\Phi_{i_1\dots i_\Theta}$ being a smooth function which does not depend on the variables $\tau_{i_j}$. Now, a computation of the $\xi$-derivatives of $\, ^{(i_1\dots i_\Theta)} \tilde \psi ^{wk}$ in any of the $\alpha^{(i_\Theta)}$-charts shows that $\, ^{(i_1\dots i_\Theta)} \tilde \psi ^{wk}$ has no critical points  there. 
 By the non-stationary phase theorem, see H\"{o}rmander \cite{hoermanderI}, Theorem 7.7.1, one then computes for arbitrary $\tilde N \in \N$
 \begin{align*}
| \tilde I_{i_1 \dots i_\Theta}(\mu)| \leq c_7 \mu^{\tilde N} \int_{\epsilon<|\tau_{i_j}| < T} \prod_{j=1} ^\Theta |\tau_{i_j}|^{c^{(i_j)}+\sum_r^j d^{(i_r)}-1-\tilde N}d\tau + c_8 \epsilon^{\Theta (\kappa+1)} \leq c_9 \max\mklm{\mu^{\tilde N},\mu^{\kappa+1}},
 \end{align*}
 where we took $\epsilon=\mu^{1/\Theta}$. Choosing $\tilde N$ large enough, we conclude that
 \bqn
 | \tilde I_{i_1 \dots i_\Theta}(\mu)| =O(\mu^{\kappa+1}).
 \eqn
  As a consequence of this we see that, up to terms of order $O(\mu^{\kappa+1})$,  $I(\mu)$ can be written as a sum
\begin{align*}
I(\mu)&=\sum_{k<L} I_k(\mu)+I_L(\mu) =\sum_{k<l<L} I_{kl}(\mu) + \sum_{k<L} I_{kL}(\mu)  +I_L(\mu)\\
&=\sum _N\sum_{i_1<\dots< i_N<  i_{N+1} =L} I_{i_1\dots i_N}( \mu)+\sum_M \sum_{i_1<\dots<i_M<i_{M+1} \not=L} I_{i_1\dots i_M L}( \mu),
\end{align*}
where the first term in the last line is a sum to be taken over all the indices $i_1,\dots, i_N$ corresponding to all possible isotropy branches of the form $(H_{i_1},\dots,  (H_{i_{N}}), (H_{i_{N+1}})=(H_L))$ of varying length $N$, while the second term is a sum over all indices $i_1, \dots, i_M$ corresponding to branches  $(H_{i_1},\dots, (H_{i_{M}}), (H_{i_{M+1}})\not=(H_L))$ of arbitrary length $M$. The asymptotic behavior of the integrals $I_{i_1\dots i_N}(\mu)$ has been determined in the previous section, and it is not difficult to see that the integrals $I_{i_1\dots i_ML}$ have analogous asymptotic descriptions. We are now ready to state and prove the main result of this paper.
\begin{theorem}
Let $G$ be  a compact, connected Lie group $G$ with Lie algebra $\g$, acting orthogonally on Euclidean space $\rn$, and define
\bqn
I(\mu)=   \int _{T^\ast \rn}  \int_{\g} e^{i \psi( x,\xi,X)/\mu }   a(x,\xi,X)  \d X \d\xi \d x ,   \qquad \mu >0,
\eqn 
where the phase function 
\bqn 
\psi(x,\xi,X)=\mathbb{J}(x,\xi)(X)=\eklm{ X x,\xi}
\eqn
is given by the moment map $\mathbb{J}:T^\ast \rn \rightarrow \g^\ast$ of the underlying Hamiltonian action, and $dxd\xi$, $dX$ are Lebesgue measures in $T^\ast \rn$,  and $\g$, respectively, and $a \in \CT( T^\ast \rn \times \g)$.   Then 
\bqn 
I(\mu) = (2\pi\mu)^\kappa L_0 + O(\mu^{\kappa+1}),  \qquad \mu \to 0^+.
\eqn
Here $\kappa$ is the dimension of an orbit of principal type in $\rn$, and the leading coefficient is given by 
\bq
\label{eq:L0}
L_0=\int_{\mathrm{Reg}\, \mathcal{C}} \frac {a(x,\xi,X)}{|\mathrm{Hess}  \, \psi(x,\xi,X)_{N_{(x,\xi,X)}\mathrm{Reg}\, \mathcal{C}}|^{1/2}} \d(\mathrm{Reg}\, \mathcal{C})(x,\xi,X),
\eq
where $\mathrm{Reg}\, \mathcal{C}$ denotes the regular part of the critical set $\mathcal{C}=\Crit(\psi)$ of $\psi$. 
In particular, the integral over $\mathrm{Reg}\, \mathcal{C}$ exists.
\end{theorem}
\begin{remark}
Note that Equation \eqref{eq:L0} in particular means that the obtained asymptotic expansion for $I(\mu)$ is independent of the explicit partial resolution we used.
\end{remark}
\begin{proof}
By our previous considerations, one has 
\bqn 
I(\mu) = (2\pi\mu)^\kappa L_0 + O(\mu^{\kappa+1}),  \qquad \mu \to 0^+,
\eqn
where $L_0$ is given as a sum of integrals of the form \eqref{eq:L}. It therefore remains to show the equality \eqref{eq:L0}. For this let us introduce first certain cut-off functions for the singular part $\Sing \Omega$ of $\Omega$. Thus, let $K$  be a compact subset in $\R^{2n}$, $\epsilon >0$, and denote by $v_\epsilon$ the characteristic function of the set
 \bqn
 (\Sing \Omega \cap K)_{2\epsilon}=\mklm{z \in \R^{2n}: |z-z'| <2 \epsilon \text{ for some } z' \in \Sing \Omega\cap K}.
 \eqn
 Consider further the unit ball $B_1$ in $\R^{2n}$, together with a function $\iota \in \CT(B_1)$ with $\int \iota dz=1$, and set $\iota_\epsilon(z) =\epsilon^{-2n} \iota(z/\epsilon)$. Clearly  $\int \iota_\epsilon dz =1$, $\supp \iota_\epsilon \subset B_\epsilon$, and we define
 \bq
 \label{eq:u}
 u_\epsilon=v_\epsilon \ast \iota_\epsilon.
 \eq
One can then show that $u_\epsilon \in \CT((\Sing \Omega \cap K)_{3\epsilon})$, and $u_\epsilon =1$ on $(\Sing \Omega\cap K)_\epsilon$, together with
\bqn
|\gd^\alpha_z u _\epsilon|\leq C_\alpha \epsilon^{-|\alpha|},
\eqn
where $C_\alpha$ is a constant which depends only on $\alpha$ and $n$, see H\"{o}rmander \cite{hoermanderI},  Theorem 1.4.1.\\\
Next,  we shall prove
\begin{lemma}\label{lemlim}
Let $a \in \CT(\R^{2n}\times \g)$, and $K$ be such that $\supp a \subset K$. Then the  limit 
\bq
\label{eq:MC}
\lim_{\eps \to 0}  \int_{\Reg \mathcal{C}}\frac{ [a (1-u_\eps)] (x,\xi,X)}{|\det  \, \psi'' (x,\xi,X)_{|N_{(x,\xi,X)}\Reg \mathcal{C}} |^{1/2}} d(\Reg \mathcal{C})(x,\xi,X)
\eq
exists and is equal to $L_0$, where $d(\Reg {\mathcal{C}})$ is the induced Riemannian measure on $\Reg {\mathcal{C}}$.
\end{lemma}
\begin{proof}
With $u_\eps$ as in Equation \eqref{eq:u}, let us define
\bqn
I_\eps(\mu)=\int_{T^\ast \rn} \int_{\g}  e^{\frac{i}{\mu} \psi(x,\xi,X) }  [a(1-u_\eps)](x,\xi,X) \, dX \, d\xi \, dx.
\eqn
Since $(x,\xi,X) \in \Sing {\mathcal{C}}$ implies $ (x,\xi) \in \Sing \Omega$, a direct application of the generalized theorem of the stationary phase for fixed $\eps >0$ gives
\bq
\label{eq:asympt}
| I_\eps(\mu)- (2\pi \mu)^\kappa L_0(\eps) | \leq C_\eps \mu^{\kappa+1},
\eq
where $C_\eps>0$ is a constant depending only on $\eps$, and
\bqn
L_0(\eps)=  \int_{\Reg {\mathcal{C}}}\frac{ [a(1-u_\eps)] (x,\xi,X)}{|\det  \, \psi'' (x,\xi,X)_{|N_{(x,\xi,X)}\Reg {\mathcal{C}}} |^{1/2}} d(\Reg {\mathcal{C}})(x,\xi,X).
\eqn
On the other hand, applying our previous considerations to  $I_\eps(\mu)$ instead of $I(\mu)$, we obtain again an asymptotic expansion of the form \eqref{eq:asympt} for $I_\eps(\mu)$, where now, the first coefficient is given by a sum of integrals of the form \eqref{eq:L} with $a$ replaced by $a(1-u_\eps)$. Since the first term in the asymptotic expansion \eqref{eq:asympt} is uniquely determined, the two expressions for $L_0(\eps)$ must be identical. The statement of the lemma now follows by  the Lebesgue theorem on bounded convergence.
\end{proof}
Note that existence of the limit in  \eqref{eq:MC} has been established by  partially resolving the  singularities of the critical set $\mathcal{C}$, the corresponding limit being given by $L_0$. Let now $a ^+ \in \CT(\R^{2n} \times \g), R^+)$. Since $|u_\epsilon| \leq 1$, the lemma of Fatou implies that 
\bqn 
  \int_{\Reg \mathcal{C}} \lim_{\eps \to 0}  \frac{ [a^+ (1-u_\eps)] (x,\xi,X)}{|\det  \, \psi'' (x,\xi,X)_{|N_{(x,\xi,X)}\Reg \mathcal{C}} |^{1/2}} d(\Reg \mathcal{C})(x,\xi,X)
\eqn
is mayorized by the limit  \eqref{eq:MC}, with $a$ replaced by $a^+$. Lemma \ref{lemlim} then implies that
\bqn 
  \int_{\Reg \mathcal{C}} \frac{ a^+ (x,\xi,X)}{|\det  \, \psi'' (x,\xi,X)_{|N_{(x,\xi,X)}\Reg \mathcal{C}} |^{1/2}} d(\Reg \mathcal{C})(x,\xi,X) < \infty.
\eqn
Choosing now $a^+$ to be equal  $1$ on the compact set $K$  in which $a$ was supported, and applying the theorem of Lebesgue on bounded convergence to the limit \eqref{eq:MC}, we obtain Equation \eqref{eq:L0}. 
\end{proof}

\providecommand{\bysame}{\leavevmode\hbox to3em{\hrulefill}\thinspace}
\providecommand{\MR}{\relax\ifhmode\unskip\space\fi MR }
\providecommand{\MRhref}[2]{%
  \href{http://www.ams.org/mathscinet-getitem?mr=#1}{#2}
}
\providecommand{\href}[2]{#2}


\end{document}